\documentclass[english,reqno]{amsart}
\usepackage{color}
\usepackage{amssymb}
\usepackage{amsmath}
\usepackage{amsthm}
\usepackage{amscd}
\usepackage{mathrsfs}
\usepackage{tikz}
\usepackage{bm}
\theoremstyle{definition}
\newtheorem{dfn}{Definition}[section]
\theoremstyle{plain}
\newtheorem{thm}[dfn]{Theorem}

\newtheorem{cor}[dfn]{Corollary}
\newtheorem{lm}[dfn]{Lemma}
\theoremstyle{remark}

\newtheorem{rem}[dfn]{Remark}
\numberwithin{equation}{section}

\usepackage[linktocpage,colorlinks,linkcolor=blue,anchorcolor=blue,citecolor=blue,urlcolor=black,pagebackref]{hyperref}

\newcommand{\quotient}[2]{
\mathchoice{  \text{\raise1ex\hbox{$#1$}\!\Big/\!\lower1ex\hbox{$#2$}} }
                  {  \text{\raise1pt\hbox{$#1$}\big/\lower1pt\hbox{$#2$}} }
                  {  {#1}\,/\,{#2}  }
                  {  {#1}\,/\,{#2}  }
}
\title{The strong Haagerup inequality for $q$-circular systems}
\author{Todd Kemp}
\address{Department of Mathematics, University of California, San Diego 9500 Gilman Drive \#0112 La Jolla, CA  92093-0112, USA}
\email{tkemp@ucsd.edu}
\thanks{T. Kemp acknowledges support from NSF grants DMS-2055340 and DMS-2400246.}

\author{Akihiro Miyagawa}
\address{Department of Mathematics, University of California, San Diego 9500 Gilman Drive \#0112 La Jolla, CA  92093-0112, USA}
\email{amiyagwa@ucsd.edu}
\thanks{A. Miyagawa acknowledges support from JSPS Overseas Research Fellowship and NSF grant DMS-2055340.}

\begin{document}
\begin{abstract}
  Together with Speicher, in 2007 the first author proved the {\em strong Haagerup inequality} for operator norms of homogeneous holomorphic polynomials in freely independent $\mathscr{R}$-diagonal elements (including in particular circular random variables); the inequality improved the bound from the original Haagerup inequality to grow with $\sqrt{n}$, rather than linearly in $n$, on homogeneous polynomials of degree $n$. In this paper, we prove a similar inequality for $q$-circular systems for $|q|<1$, generalizing the free case when $q=0$.  In particular, we prove the strong Haagerup inequality for systems exhibiting neither free independence nor $\mathscr{R}$-diagonality.  As an application, we prove a {\em strong ultracontractivity} theorem for the $q$-Ornstein--Uhlenbeck semigroup, and prove sharp rates for the Haagerup and ultracontractive inequalities.
\end{abstract}

\maketitle

\section{Introduction}
In 1979, Haagerup proved a landmark inequality bounding the operator norm of a homogeneous polynomial in the generators of a free group by its $\ell^2$ norm.  The precise result published in \cite{MR0520930}, now known as the {\em Haagerup inequality}, is as follows: in the reduced $\mathrm{C}^*$-algebra $\mathrm{C}^*_{\mathrm{red}}(\mathbb{F}_d)$ of the free group $\mathbb{F}_d$ on $d\ge 2$ generators,
\[\left\| \sum_{ |g|= n}\alpha_g \lambda(g) \right\|_{\mathrm{C}^*_{\mathrm{red}}(\mathbb{F}_d)}\le (n+1)\left\|\sum_{|g|= n}\alpha_g \lambda(g) \right\|_{\ell^2(\mathbb{F}_d)}\]
where $|g|$ is the length of $g\in \mathbb{F}_d$ as a reduced word.

Equipping $\mathrm{C}_{\mathrm{red}}^\ast(\mathbb{F}_d)$ with the tracial state induced by the indicator of the identity, the $\mathrm{C}^\ast$-algebra becomes a $\mathrm{C}^\ast$-probability space, and the generators $g_1,\ldots,g_d$ of the group yield operators $\lambda(g_1),\ldots,\lambda(g_d)$ that are freely independent (each possessing the Haar unitary distribution).  The result may therefore be interpreted in a self-contained free probabilistic context, as follows.

Let $u_1,\ldots,u_d$ be freely independent Haar unitary elements of a $\mathrm{C}^\ast$-probability space.  For each $n\in\mathbb{N}$, if $w=w_1w_2\cdots w_n$ is a word of length $n$ in $\{\pm 1,\ldots,\pm d\}^n$, denote by $u_w = u_{w_1}\cdots u_{w_n}$ the product of the Haar unitaries indexed by the word $w$, where $u_{-k} = u_k^{-1}$.  Then for any scalars $\alpha_w\in\mathbb{C}$,
\begin{equation} \label{eq:Haagerup} \left\|\sum_{|w|=n} \alpha_w u_w\right\| \le (n+1)\left\|\sum_{|w|=n} \alpha_w u_w\right\|_2 \end{equation}
where $\|\,\cdot\,\|$ denotes the $\mathrm{C}^\ast$-norm, $\|\,\cdot\,\|_2$ denotes the noncommutative $L^2$-norm of the $\mathrm{C}^\ast$-probability space, $|w|=n$ means that the word $w$ has {\em reduced} length $n$, meaning that no two adjacent indices in $w$ are negatives of each other.

Sums of the form in Haagerup's inequality are homogeneous polynomials in the given unitary variables {\em and their inverses}, or equivalently their adjoints.  In \cite{MR2353703}, the first author and Speicher considered the case of homogeneous {\em holomorphic} polynomials, where no adjoints are allowed.  This simplifies the set of words in the sum, since cancelations are no longer possible (all non-trivial words are already reduced).  This turns out to have a profound effect on the inequality: instead of growing linearly with $n$, the constant grows with $\sqrt{n}$ instead.  \cite{MR2353703} proved this not only for the original context of freely independent Haar unitary elements, but more generally for freely independent $\mathscr{R}$-diagonal elements introduced by Nica and Speicher in \cite{MR1426839}.

\begin{thm}[\cite{MR2353703}] \label{thm.KS} Let $a_1,\ldots,a_d$ be freely independent identically distributed $\mathscr{R}$-diagonal elements in a tracial $\mathrm{C}^\ast$-probability space.  There is a constant $C_{a_1}$, independent of $d$, so that, for each $n\in\mathbb{N}$,
\[\left\| \sum_{ |w|= n}\alpha_w a_w \right\|\le C_{a_1} \sqrt{n+1}\left\|\sum_{|w|= n}\alpha_w a_w \right\|_{2}.\]
\end{thm}
The inequality in Theorem \ref{thm.KS} is called the {\em strong Haagerup inequality}.  The collection of $\mathscr{R}$-diagonal elements forms a large family of non-self-adjoint random variables including Voiculescu's circular elements and Haar unitaries.  They are precisely the large-dimension limits of bi-invariant random matrix ensembles (meaning with law invariant under multiplication on the left or the right with a fixed unitary); in free probability terms, they can be characterized by the property that $ua$ and $au$ have the same $\ast$-distribution as $a$ for any Haar unitary $u$ free from $a$.  Combinatorially, this amounts to a collections of alternating symmetries of their joint free cumulants; as such, the proof of the strong Haagerup inequality in \cite{MR2353703} is based on the moment-cumulant formula in free probability.  The fact that the homogeneous polynomials involved are holomorphic sets up alternating patterns in their moments, compatible with the $\mathscr{R}$-diagonal structure; in the general case involving adjoints as well as in Haagerup's original case \eqref{eq:Haagerup}, it is unknown whether the result extend to general $\mathscr{R}$-diagonal elements or beyond.  The one exception is the case of freely independent semicircular elements, where a version of \eqref{eq:Haagerup} was shown to hold in \cite{MR1811255}, but where `homogeneous' must be interpreted differently, in terms of Tchebyshev polynomials in each variable rather than monomials.  We explain this in detail in Section \ref{preliminary}.

There have been some generalizations of Kemp--Speicher's strong Haagerup inequality. In the paper \cite{MR2557731}, de la Salle proved that freely independent $\mathscr{R}$-diagonal elements also satisfy the analog of this inequality for operator coefficients polynomials like $ \sum_{ |w|= n}\alpha_w \otimes x_w $. On the other hand, Brannan \cite[Theorem 1.4]{MR2892462} generalized the strong Haagerup inequality for non-commutative random variables with some invariant properties that do not require the variables to be freely independent. In particular, he proved the strong Haagerup inequality for the free unitary quantum group $U_N^{+}$ (see also \cite{MR4432279} for non-Kac type orthogonal quantum groups). 

In this paper, we prove that this type of estimate also holds for $q$-circular systems $c_1^{(q)},\ldots,c_d^{(q)}$ (see the discussion below and Section \ref{preliminary}), which are not freely independent and do not satisfy assumptions in Brannan's paper \cite[Theorem 1.4.]{MR2892462} (more precisely, the $q$-circular system is not invariant under free complexification). Namely, our main result is the following.
\begin{thm}\label{Mainresult}
    For $d\in\mathbb{N}$ and $-1<q<1$, let $\{c_1^{(q)},\ldots,c_d^{(q)}\}$ be a $q$-circular system. there exists some $A=A(|q|)$ (independent of $d$) such that, for any $n \in \mathbb{N}$ and coefficients $\{\alpha_w\}_{w\in[d]^*}$ in $\mathbb{C}$,
    \[
    \left\|\sum_{|w|=n}\alpha_w c_w^{(q)}\right\|\le A\sqrt{n+1}\left\|\sum_{|w|=n}\alpha_w c_w^{(q)}\right\|_2.\]
\end{thm}
The concept of a $q$-circular system was introduced by Mingo and Nica \cite{MR1871395} as a $q$-deformation of a circular system ($q=0$) based on the work on the $q$-canonical commutation relations \cite{MR0260352}, \cite{MR1105428}. These $q$-relations interpolate between the Bosonic $(q=1)$ CCR (canonical commutation relations) and Fermionic $(q=-1)$ CAR (canonical anticommutation relations) in quantum field theory. From this perspective, the algebra of holomorphic polynomials $\sum_{w}\alpha_w c_w^{(q)}$ has been studied in terms of a $q$-analog of the Segal-Bergman space (see \cite{MR2174419} and \cite{MR3802726}). Regarding our main result, if we replace $\sqrt{n+1}$ with $n+1$, the inequality follows from Bo\.{z}ejko's Haagerup inequality \cite{MR1811255} for $q$-Gaussian systems together with the $q$-Segal--Bargmann isometry in \cite{MR2174419}.  In addition, the second author proved that the operator norm of $\sum_{|w|=n}\alpha_w c_w^{(q)}$ is continuous as a function of $q$ \cite{MR4685342}. However, these results did not give any indication whether the strong Haagerup inequality holds in the $q$-circular context, until now.

There is a heuristic observation of this $q$-deformation that the case $-1<q<1$ shares the same properties with the case $q=0$ which can be described by free probability. This is supported by the isomorphism of $q$-Gaussian von Neumann algebras \cite{MR3251831} and $q$-Cuntz-Toeplitz algebras \cite{MR1225964},\cite{MR4580530},  hypercontractivity and ultracontractivity of $q$-Ornstein-Uhlenbeck semigroup \cite{MR1462754}, \cite{MR2174419},\cite{MR1811255},...etc. Our main result of the strong Haagerup inequality also follows this pattern.

While previous works on the strong Haagerup inequality are based on combinatorial arguments of joint moments, we use several norm inequalities for creations and annihilation operators proved by Bo\.{z}ejko \cite{MR1811255}. He used these inequalities to prove the Haagerup inequality for $q$-Gaussian systems giving us some insights to prove the strong Haagerup inequality.
As an application, our results give us bounds on the joint moments of a $q$-circular system without computing combinatorial moments directly. Since $\|(c^{(q)})^n\|_{2m}\le \|(c^{(q)})^n\|$ for any $m \in \mathbb{N}$, by the formula of moments of a $q$-circular random variable $c^{(q)}$ \cite[Definition 1.2]{MR1871395}, our main result implies
\[\sum_{\pi \in P_{\ast-1}(\underbrace{\ast^n,\ 1^n,\ldots,\ast^n,\ 1^n}_{2mn\ \mathrm{vertices}})}q^{\mathrm{cr}(\pi)}\le A^{2m}(n+1)^m{([n]_q!)}^{m}\]
where $P_{\ast-1}(\ast^n,\ 1^n,\ldots,\ast^n,\ 1^n)$ is the set of pair partitions on the $2mn$ vertices alternately aligned with $n$ $\ast$s and $n$ $1$s in which every block connects 
 a $\ast$ with a $1$ and $\mathrm{cr}(\pi)$ denotes the number of crossings of $\pi$. When $q=0$, the left-hand side of the inequality is equal to the Fuss-Catalan number $\frac{1}{m}\binom{m(n+1)}{m-1}$ (see \cite[Corollary 3.2]{MR2353703}). 

Another way to encode the $L^2-L^\infty$ bound (and the growth of the constant with $n$) in Theorem \ref{Mainresult} is to consider more general noncommutative holomorphic polynomials in the circular variables
\[ h(\mathbf{c}^{(q)}) = f(c_1^{(q)},\ldots,c_d^{(q)}) = \sum_{w\in[d]^\ast} \alpha_w c_w^{(q)} \]
where the sum is finite, say (i.e.\ with $\alpha_w=0$ for all but finitely-many words $w$).  To encode the strong Haagerup inequality for homogeneous polynomials, we can consider the action of the dilation semigroup $D_t h(\mathbf{c}^{(q)}) = h(e^{-t}\mathbf{c}^{(q)}) = f(e^{-t}c^{(q)}_1,\ldots,e^{-t}c^{(q)}_d)$.  In other words, for a homogeneous sum of degree $n$,
\[ D_t\left(\sum_{|w|=n} \alpha_w c_w^{(q)}\right) = e^{-nt} \sum_{|w|=n} \alpha_w c_w^{(q)}. \]
The main estimates of Theorem \ref{Mainresult} prove an $L^2-L^\infty$ bound for the dilation semigroup: for some constant $\beta=\beta(|q|)$,
\begin{equation} \label{eq:dilation.semigroup} \|D_t h\| \le \frac{\beta}{t} \|h\|_2 \qquad \forall h\in\mathbb{C}\langle c_1^{(q)},\ldots,c_d^{(q)}\rangle.\end{equation}
This estimate falls under the name {\bf ultracontractivity}; the rate $1/t$ of blow-up at $t=0$ relates directly to the `strong' in our `strong Haagerup inequality' with rate $\sqrt{n+1}$ rather than $n+1$ comparing $L^2$ and operator norms of degree-$n$ homogeneous polynomials.  Section \ref{sect.ultra} below details the strong ultracontractivity theorem in \eqref{eq:dilation.semigroup} in a larger context (beyond polynomials in the realm of some natural holomorphic Hilbert spaces), and moreover proves that this estimate is sharp: $\|D_t\colon L^2\to L^\infty\|$ is precisely of order $1/t$ (bounded above and below).  The tools used to prove this also yield a simple corollary (and self-contained proof) that our main result in Theorem \ref{Mainresult} is also sharp: even in the case $d=1$, the operator norm of a degree-$n$ homogeneous holomorphic polynomial in a $q$-circular variable is bounded {\em below} by a constant times $\sqrt{n+1}$ times its $L^2$-norm.  See Theorems \ref{thm.ultra} and \ref{cor.ultra} below.

Our present approach leaves open several possibilities of other generalizations of the strong Haagerup inequality. We expect that one can show the strong Haagerup inequality for the mixed $q_{ij}$ and twisted relations since they satisfy the Haagerup inequality (see \cite{MR2091676} and \cite{MR2180382}). We also expect that one can show the strong Haagerup inequality for $q$-circular system with operator coefficients like de la Salle's results in \cite{MR2557731}.  We leave these for future work.

The remainder of this introductory section is devoted to a simplified case of our general argument in the case $q=0$.  That is: we provide a new proof of the circular case of the strong Haagerup inequality in \cite{MR2353703}, using new methods that generalize to the $q$-circular setting to prove Theorem \ref{Mainresult} in Section \ref{stronghaag} below.  Following that, Section \ref{preliminary} gives the essential background (Fock spaces, construction of $q$-circular systems, notation and results in \cite{MR1811255}, and some new preliminary computations lemmas) to set up our proof of the strong Haagerup inequality for $q$-circular systems.  The proof of that main result is the content of Section \ref{stronghaag}, a significant portion of which is devoted to the proof of the technical Lemma \ref{KeyLem} which is applied together with an argument mirroring the above one in the $q=0$ case to finish the proof of the $q$-circular strong Haagerup inequality.

\subsection{Another approach to the strong Haagerup inequality for free circular system}\label{freecase}
Here, we explain our approach to the strong Haagerup inequality in the $q=0$ case, i.e.\ for a free circular system, where there are a number of simplifications that help elucidate the core ideas of our methods.  This proof produces a non-optimal constant compared to the result in \cite{MR2353703}, but does produce the sharp $\sqrt{n}$ growth with the degree $n$ of the holomorphic polynomials.

We refer to Section \ref{preliminary} for our notations. First of all, our free circular system $\{c_i\}_{i=1}^d$ is realized on the full Fock space $\mathcal{F}_0(H)=\mathbb{C}\Omega \oplus \bigoplus_{n=1}^{\infty}H^{\otimes n}$ as the sum of left creation and annihilation operators $c_i =a_i + a_{\overline{i}}$ which satisfy $a_i^*a_j=\delta_{i,j}$ ($i,j\in \{1,\ldots,d\}\sqcup \{\overline{1},\ldots,\overline{d}\}$). 
Then, we expand products of free circular random variables with respect to left creation and annihilation operators (see Lemma \ref{qCircular}), and we have 
\[\sum_{|w|=n}\alpha_w c_w=\sum_{k=0}^n\sum_{|u|=k,|v|=n-k} \alpha_{uv} a_u a_{\overline{v}}^*=\sum_{k=0}^n A_k\]
where $A_k= \sum_{|u|=k,|v|=n-k} \alpha_{uv} a_u a_{\overline{v}}^*$. Note that we have
\[ \left\|\sum_{|w|=n}\alpha_w c_w\right\|_2^2=\left\|\sum_{|w|=n}\alpha_w e_w\right\|_{\mathcal{F}_0(H)}^2=\sum_{|w|=n}|\alpha_w|^2.\]
For a technical reason, we remove the term of $k=0$ and prove the existence of a constant $A>0$ such that $\|\sum_{k=1}^n A_k\|^2\le A n \|\sum_{|w|=n} \alpha_w c_w\|_2^2 $ for any $\{\alpha_w\}_{|w|=n}$. We have the identity for $\mathrm{C}^*$-algebras 
$\|\sum_{k=1}^n A_k\|^2=\|\sum_{k=1}^n A_k^*\sum_{l=1}^n A_l\|$ and expand the sum. Then by using the triangle inequality and the identity $\|X^*\|=\|X\|$, we have
\begin{align*}
    \left\|\sum_{k=1}^n A_k\right\|^2 &= \left\| \sum_{k=1}^n A_k^* \sum_{l=1}^nA_l\right\|\\
    &\le \left\|\sum_{k=1}^n A_k^*A_k\right\|+\left\|\sum_{k=1}^{n-1}\sum_{l=k+1}^n A_k^*A_l\right\|+\left\|\sum_{k=2}^{n}\sum_{l=1}^{k-1}A_k^*A_l \right\|\\
    &= \left\|\sum_{k=1}^n A_k^*A_k\right\|+2\left\|\sum_{k=1}^{n-1}\sum_{l=k+1}^n A_k^*A_l\right\|.
\end{align*}
It is known that $\|A_k\|\le \|\sum_{|w|=n}\alpha_w c_w\|_{2}$ for $0\le k \le n$ (see \cite[Proposition 2.1]{MR1811255} or Lemma \ref{norm_of_prodcuts} in this paper). Therefore, by the triangle inequality, we have
\[ \left\|\sum_{k=1}^n A_k^*A_k\right\|\le \sum_{k=1}^n \|A_k\|^2\le n \left\|\sum_{|w|=n}\alpha_w c_w\right\|_{2}^2 \]
By using the relation $a_i^* a_j = \delta_{i,j}$ ($i,j\in[d,\overline{d}]$), we can cancel some terms in $\left\|\sum_{k=1}^{n-1}\sum_{l=k+1}^n A_k^*A_l\right\|^2$ and this is equal to
    \begin{align*}
       &\left\|\sum_{k=1}^{n-1}\sum_{l=k+1}^n \sum_{|u_1|=k,|v_1|=n-k}\sum_{|u_2|=l,|v_2|=n-l}\overline{\alpha_{u_1v_1}}\alpha_{u_2v_2}(a_{\overline{v_1}}^*)^*(a_{u_1})^*a_{u_2}a_{\overline{v_2}}^*\right\|^2 \\
       &=\left\|\sum_{k=1}^{n-1}\sum_{l=k+1}^n \sum_{\substack{|u_1|=k,|v_1|=n-k\\|u_2|=l-k,|v_2|=n-l}}\overline{\alpha_{u_1v_1}}\alpha_{u_1u_2v_2}a_{\overline{v_1}^*}a_{u_2}a_{\overline{v_2}}^*\right\|^2
    \end{align*} 
    where we use $(a_{u_1})^*a_{u_2}=\delta_{u_1,x}a_{u'_2}$ for the decomposition $u_2=xu'_2$ with $|x|=k$ and $|u'_2|=l-k$ ($\delta_{u_1,x}=1$ if $u_1=x$ and $0$ otherwise) and we replace $u'_2$ by $u_2$.
Since $a_{\overline{i}}^*a_j = 0$ ($i,j\in [d]$), for $|v'_1|=n-k'$, $|u'_2|=l'$, $|v_1|=n-k$, $|u_2|=l$ with $k\neq k'$, we have
\[ (a_{\overline{v'_1}^*}a_{u'_2})^* a_{\overline{v_1}^*}a_{u_2}=0\]
and if $k=k'$,
\[ (a_{\overline{v'_1}^*}a_{u'_2})^* a_{\overline{v_1}^*}a_{u_2}= \delta_{v'_1,v_1} (a_{u'_2})^*a_{u_2}.\]
This implies 
    \begin{align*}
        &\left\|\sum_{k=1}^{n-1}\sum_{l=k+1}^n \sum_{\substack{|u_1|=k,|v_1|=n-k\\ |u_2|=l-k,|v_2|=n-l}}\overline{\alpha_{u_1v_1}}\alpha_{u_1u_2v_2}a_{\overline{v_1}^*}a_{u_2}a_{\overline{v_2}}^*\right\|^2\\
        &=\sum_{k=1}^{n-1}\sum_{|v_1|=n-k}\left\|\sum_{l=1}^{n-k} \sum_{|u_2|=l,|v_2|=n-k-l} \left(\sum_{|u_1|=k}\overline{\alpha_{u_1v_1}}\alpha_{u_1u_2v_2}\right)a_{u_2}a_{\overline{v_2}}^*\right\|^2.
        \end{align*}
We now construct a (strong) induction using the above inequality as the inductive step.  For $1\le k\le n-1$, presume we have proved that $ \|\sum_{l=1}^{n-k} A_l\|^2 \le A (n-k) \|\sum_{|w|=n-k}\alpha_w c_w\|_{2}^2$ for any $1\le k\le n-1$ (here $\alpha_w$ is replaced with $ \sum_{|u_1|=k}\overline{\alpha_{u_1v_1}}\alpha_{u_1w}$); the base case $k=n-1$ follows from the inequality $\|A_1\|^2 \le \|\sum_{|w|=1}\alpha_w c_w\|_{2}^2$ which we mentioned. Then we have      
        \begin{align*}
        &\sum_{k=1}^{n-1}\sum_{|v_1|=n-k}\left\|\sum_{l=1}^{n-k} \sum_{|u_2|=l,|v_2|=n-k-l} \left(\sum_{|u_1|=k}\overline{\alpha_{u_1v_1}}\alpha_{u_1u_2v_2}\right)a_{u_2}a_{\overline{v_2}}^*\right\|^2\\
         &\le A \sum_{k=1}^{n-1} (n-k) \sum_{|v_1|=n-k}\left\|\sum_{|w|=n-k} \left(\sum_{|u_1|=k}\overline{\alpha_{u_1v_1}}\alpha_{u_1w}\right)c_w\right\|^2_{2}\\
    \end{align*}
    Since $\|\sum_{|w|=n}\alpha_w c_w\|_2^2=\sum_{|w|=n}|\alpha_w|^2$, we have
    \[ \left\|\sum_{|w|=n-k} \left(\sum_{|u_1|=k}\overline{\alpha_{u_1v_1}}\alpha_{u_1w}\right)c_w\right\|^2_{2} =\sum_{|w|=n-k} \left|\sum_{|u_1|=k}\overline{\alpha_{u_1v_1}}\alpha_{u_1w}\right|^2. \]
    Then, by using Cauchy-Schwarz inequality, we have
    \begin{align*}
       & A \sum_{k=1}^{n-1} (n-k) \sum_{|v_1|=n-k}\sum_{|w|=n-k} \left|\sum_{|u_1|=k}\overline{\alpha_{u_1v_1}}\alpha_{u_1w}\right|^2\\
       &\le  A \sum_{k=1}^{n-1} (n-k) \sum_{|v_1|=n-k}\sum_{|w|=n-k} \sum_{|u_1|=k}|\overline{\alpha_{u_1v_1}}|^2\sum_{|u_1|=k}|\alpha_{u_1w}|^2\\
       &\le \frac{An^2}{2} \left\|\sum_{|w|=n}\alpha_w c_w\right\|_2^4.
    \end{align*}
    Therefore, if $A$ satisfies $\sqrt{2A}+1\le A$, we have
    \begin{align*}
    \left\|\sum_{k=1}^nA_k\right\|^2 &\le  \left\|\sum_{k=1}^n A_k^*A_k\right\|+2\sqrt{\left\|\sum_{k=1}^{n-1}\sum_{l=k+1}^n A_k^*A_l\right\|^2}
    \\&\le \left(\sqrt{2A}+1\right)n \left\|\sum_{|w|=n}\alpha_w c_w\right\|_{2}^2\\
     &\le A n \left\|\sum_{|w|=n}\alpha_w c_w\right\|_{2}^2.
     \end{align*}
     To obtain the strong Haagerup inequality, we use the triangle inequality and $\|A_0\|\le \|\sum_{|w|=n}\alpha_w c_w\|_2$, and we have
     \begin{align*}
         \left\|\sum_{|w|=n}\alpha_w c_w\right\| &\le \|A_0\|+\left\|\sum_{k=1}^nA_k\right\| \\
         &\le (\sqrt{An}+1) \left\|\sum_{|w|=n}\alpha_w c_w\right\|_2 \\
         &\le A'\sqrt{n+1}\left\|\sum_{|w|=n}\alpha_w c_w\right\|_2 
     \end{align*}
     for some $A'>0$ ($A'=\sqrt{2A}$, for example).

     This argument was facilitated heavily by the commutation relation $a_i^\ast a_j = \delta_{i,j}$, which was used repeatedly above to simplify sums.  In the $q\ne 0$ case, the commutation relation $a_i^\ast a_j - q a_j a_i^\ast = \delta_{i,j}$ will play a similar role, but leads to many more complicated terms.  The overall flow of the above proof still stands, but much more sophisticated estimates (not just iterated use of the triangle inequality and the Cauchy-Schwarz inequality) will be required to push the analysis through.

\section{Background, Setup, and Preliminary Estimates}\label{preliminary}

The principle construct needed for our present work is the ($q$-)Fock space.  Let $H$ denote a complex Hilbert space. The algebraic Fock space over $H$ is defined by
\[\mathcal{F}_{\mathrm{alg}}(H)=\bigoplus_{n=0}^{\infty} H^{\otimes n}\]
where we take the algebraic direct sum and we set $H^{\otimes 0} = \mathbb{C}\Omega$ with a unit vector $\Omega$.  (In other words, $\mathcal{F}_{\mathrm{alg}}(H)$ is the tensor algebra generated by $H$.)

Throughout the paper, fix $-1<q<1$.  In \cite{MR1105428}, Bo\.{z}ejko and Speicher introduced the $q$-inner product of $\xi_1\otimes \cdots \otimes \xi_m \in H^{\otimes m}$ and $\eta_1 \otimes \cdots \otimes \eta_n \in H^{\otimes n}$, using the boosted inner product on $H^{\otimes m}$ composed with a ``$q$-symmetrization''
\[\langle \xi_1\otimes \cdots \otimes \xi_m, \eta_1 \otimes \cdots \otimes \eta_n \rangle_q = \delta_{m,n} \langle P^{(m)}_q\xi_1\otimes \cdots \otimes \xi_m, \eta_1 \otimes \cdots \otimes \eta_n\rangle_{H^{\otimes m}} \]
where the operator $P^{(m)}_q = P^{(m)}$ on $H^{\otimes m}$ is defined by
\begin{align*} P^{(m)}(\xi_1\otimes \cdots \otimes \xi_m) &= \sum_{\pi \in S_m} q^{\mathrm{inv}(\pi)} \xi_{\pi(1)}\otimes \cdots \otimes \xi_{\pi(m)} \\
&= \sum_{\pi \in S_m} q^{\mathrm{inv}(\pi)} \xi_{\pi^{-1}(1)}\otimes \cdots \otimes \xi_{\pi^{-1}(m)}.\end{align*}
Here the sum is taken over the symmetric group $S_m$, and $\mathrm{inv}(\pi)$ is the number of inversions in the permutation $\pi$ (the number of pairs $i<j$ for which $\pi(i)>\pi(j)$).  Note that $\mathrm{inv}(\pi) = \mathrm{inv}(\pi^{-1})$ for all $\pi\in S_m$.  Bo\.zejko and Speicher proved that, for $-1<q<1$, $P^{(m)}_q$ is a strictly positive operator, and hence the $q$-inner product is indeed a non-degenerate inner product.  The completion of $\mathcal{F}_{\mathrm{alg}}(H)$ in this inner product is called the \emph{$q$-Fock space}.

For a vector $\xi \in H$, we define \emph{creation operator} $a(\xi)$ and \emph{annihilation operator} $a^*(\xi)$ by
\begin{align*}
    a(\xi)\eta_1\otimes \cdots \otimes \eta_n &= \xi \otimes \eta_1\otimes \cdots \otimes \eta_n, \\ a(\xi)\Omega&=\xi,\\
    a^*(\xi)\eta_1\otimes \cdots \otimes \eta_n&=\sum_{k=1}^n q^{k-1}\langle \eta_k ,\xi\rangle_H \eta_1\otimes \cdots \otimes \eta_{k-1}\otimes \eta_{k+1}\cdots \otimes \eta_n, \\ a^*(\xi)\eta&=\langle \eta ,\xi\rangle_H\Omega, \\
     a^*(\xi)\Omega&=0.
\end{align*}

\begin{rem} It is much more common to let $a(\xi)$ denote the annihilation operator, in which case its adjoint $a^\ast(\xi)$ is the creation operator, just based on the first letter of the word `annihilation'.  Our convention throughout this paper is the reverse. \end{rem}

The creation and annihilation operators satisfy the {\em $q$-commutation relations}:
\[a^*(\eta)a(\xi)-qa(\xi)a^*(\eta)=\langle \xi ,\eta \rangle_H I.\]

For our present purposes, we will restrict our attention to the finite dimensional (doubled) Hilbert space $H=\mathbb{C}^d\oplus \mathbb{C}^d=H_1\oplus H_2$.  It will often be convenient to work in a fixed orthonormal basis for $H$, which we assemble from an orthonormal basis $\{e_i\}_{i=1}^d$ for $H_1$ and a separate orthonormal basis $\{e_{\overline{i}}\}_{i=1}^d$ for $H_2$.  In the special case that $\xi$ is one of the basis vectors $e_i$ or $e_{\overline{i}}$, we denote the corresponding creation and annihilation operators simply as
\[ a(e_i) = a_i \qquad a(e_{\overline{i}}) = a_{\overline{i}}
\qquad a^\ast (e_i) = a_i^\ast \qquad a^\ast(e_{\overline{i}}) = a_{\overline{i}}^\ast.\]
The $q$-commutation relations can then be reduced to the form
\[ a_i^\ast a_j - q a_j a_i^\ast = \delta_{i,j} I. \]

\emph{The $q$-circular system} $(c_1,\ldots,c_d)$ is the $d$-tuple of operators on the $q$-Fock space $\mathcal{F}_q(H)$ defined by
\[c^{(q)}_i=a(e_i)+a(e_{\overline{i}})^*=a_i+a^*_{\overline{i}}.\]
 The von Neumann algebra $\mathrm{W}^*(c_1^{(q)},\ldots,c_d^{(q)})$ generated by a $q$-circular system admits a faithful tracial state $\tau(T)=\langle T\Omega, \Omega\rangle_q$, which naturally fits in the framework of non-commutative probability spaces. We define the $L^2$-norm $\|T\|_2$ of $T \in \mathrm{W}^*(c_1^{(q)},\ldots,c_d^{(q)})$ by
 \[\|T\|_2 =\tau(T^*T)^{\frac{1}{2}}=\|T \Omega\|_{\mathcal{F}_q(H)}\]
 where $\|\cdot \|_{\mathcal{F}_q(H)}$ denotes the norm on the Hilbert space $\mathcal{F}_q(H)$.

\begin{rem}
The joint moments of a $q$-circular system were computed by Mingo and Nica \cite[Definition 1.2.]{MR1871395}. They are characterized by pair partitions connecting $c_i$ with $c_i^*$, and count the number of crossings in such pairings.
    Note that the $q$-circular system $(c_1,\ldots,c_d)$ is not freely independent and each $c_i$ is not $\mathscr{R}$-diagonal. One can see this by checking 4th moments:
     \begin{align*}
     \tau(c_1^{(q)*}c_2^{(q)*}c_1^{(q)}c_2^{(q)})&=q\\
     \tau(c_1^{(q)*}c_1^{(q)*}c_1^{(q)}c_1^{(q)})&=1+q.
     \end{align*}
     The first equality proves $c_1,c_2$ are not freely independent, and the second equality shows $\kappa_4(c_1^*,c_1^*,c_1,c_1)=q$ which implies $c_1$ is not $\mathscr{R}$-diagonal and also not invariant under free complexification introduced in \cite[Definition 2.7]{MR2892462}.
\end{rem}

\label{ast.on.words}

We also consider the set of words $[d,\overline{d}]^*$ composed of finite strings of elements in $1,\ldots,d$ and $\overline{1},\ldots,\overline{d}$ and the empty word $\Omega$. For $w \in [d,\overline{d}]^*$, $|w|$ denotes the word length of $w$ and we set $|\Omega|=0$. We define $\overline{w}=\overline{w_1}\cdot \overline{w_2}\cdots \overline{w_n}$ for $w=w_1w_2\cdots w_n \in [d]^*$ and $w^*=w_nw_{n-1}\cdots w_1$ for $w=w_1w_2\cdots w_n \in [d,\overline{d}]^*$. 
In our convention, for each $w=w_1w_2\cdots w_n\in [d,\overline{d}]^*$ and a family of operators $\{T_i\}_{i\in [d,\overline{d}]}$, we write $T_w=T_{w_1}T_{w_2}\cdots T_{w_n}$ and $T_{\Omega}=1$. We set a linear basis $\{e_w\}_{w \in [d,\overline{d}]}$ in $\mathcal{F}_{\mathrm{alg}}(H)$ by $e_w=e_{w_1}\otimes \cdots \otimes e_{w_n}$ for $w=w_1\cdots w_n$ and $e_{\Omega}=\Omega$ (ergo there is no confusion with $\Omega$ doing double duty).

We identify $\pi \in S_m$ with the permutation operator on $H^{\otimes m} $ and the permutation of letters in a word $w$ as follows: we define an operator $\pi$ on $H^{\otimes m}$ by
\[\pi (\xi_1 \otimes \cdots \otimes \xi_m)=\xi_{\pi^{-1}(1)} \otimes \cdots \otimes \xi_{\pi^{-1}(1)},\]
and we define a permutation $\pi(w)$ of a word $w=w_1\cdots w_m$ by 
\[\pi(w)=w_{\pi^{-1}(1)}\cdots w_{\pi^{-1}(m)}, \]
which defines a left action of $S_m$ on $[d,\overline{d}]^\ast$. Note that we have \[P^{(m)}=\sum_{\pi\in S_m} q^{\mathrm{inv}(\pi)} \pi^{-1} =\sum_{\pi\in S_m} q^{\mathrm{inv}(\pi)} \pi.\]

Denote by $S_k \times S_{m-k}$ the subgroup of $S_m$ such that $\pi \in S_k \times S_{m-k}$ maps $\{1,\ldots,k\}$ to $\{1,\ldots,k\}$ and $\{k+1,\ldots,m\}$ to $\{k+1,\ldots,m\}$.  We then let
\[ \quotient{S_m}{S_k \times S_{m-k}} :=\text{ the left cosets of $S_k \times S_{m-k}$ in $S_m$.} \]
In this paper, we always take the unique representative of $\sigma \in \quotient{S_m}{S_k \times S_{m-k}}$ so that $\mathrm{inv}(\sigma)$ is minimal. Such permutations can be described as a permutation $\sigma \in S_m$ such that $\sigma(1)<\sigma(2)<\cdots<\sigma(k)$ and $\sigma(k+1)<\sigma(k+2)\ldots<\sigma(m)$. If we take $\pi \in S_m$, then we have the unique factorization $\pi = \sigma (\tau_1 \times \tau_2)$ where $\sigma \in \quotient{S_m}{S_k \times S_{m-k}}$ is a permutation such that for each $1\le i \le k$, $\sigma(i)$ is the $i$-th smallest number in $\{\pi(j)\}_{j=1}^k$ and for each $k+1\le i \le m$,     $\sigma(i)$ is the $(i-k)$-th smallest number in $\{\pi(j)\}_{j=k+1}^m$. Note that under this factorization, we have $\mathrm{inv}(\pi)=\mathrm{inv}(\sigma) + \mathrm{inv}(\tau_1\times \tau_2)=\mathrm{inv}(\sigma) + \mathrm{inv}(\tau_1) +\mathrm{inv}(\tau_2)$. For example, if we take
\[\pi = \begin{pmatrix}1&2&3&4&5&6&7&8\\ 4&6&3&1&8&2&7&5\end{pmatrix}\in S_8,\]
then we have the decomposition $\pi=\sigma(\tau_1\times \tau_2)$ with
\[\sigma= \begin{pmatrix}1&2&3&4&5&6&7&8\\ 3&4&6&1&2&5&7&8\end{pmatrix}\in \quotient{S_8}{S_3\times S_5},  \]
and
\[\tau_1 = \begin{pmatrix}
    1&2&3\\ 2&3&1
\end{pmatrix}\in S_3,\quad \tau_2 =\begin{pmatrix}
    4&5&6&7&8 \\ 4&8 &5&7&6
\end{pmatrix}\in S_5,\]
where we have $\mathrm{inv}(\pi)=13=7+2+4=\mathrm{inv}(\sigma) + \mathrm{inv}(\tau_1) +\mathrm{inv}(\tau_2)$.

\begin{lm}[Theorem 2.1 in \cite{MR1811255}]\label{NCbinom}
We define an operator $R_{k,n}$ on $H^{\otimes n}$ by
\[R_{k,n}=\sum_{\pi \in \quotient{S_n}{S_k \times S_{n-k}}}q^{\mathrm{inv}(\pi)} \pi.\]
Then we have
\[R_{k,n}(P^{(k)}\otimes P^{(n-k)})=P^{(n)}=(P^{(k)}\otimes P^{(n-k)})R_{k,n}^*,\]
and
\[\|R_{k,n}\| \le C_{|q|}\]
where 
\[ C_{|q|}^{-1}=\prod_{m=1}^\infty (1-|q|^m). \]
Moreover, we have
\[ P^{(n)}\le C_{|q|}(P^{(k)}\otimes P^{(n-k)}).\]

\end{lm}
\begin{rem}
    In the one variable case, the operators $P^{(n)}$ and $R_{k,n}$ have the same roles as the $q$-factorial $[n]_q!$ and the $q$-binomial coefficient $\binom{n}{k}_q$. This lemma is the multivariable extension of $\binom{n}{k}_q[k]_q![n-k]_q!=[n]_q!$. The proof is based on the unique factorization of $\pi=\sigma (\tau_1\times \tau_2)$ with $\sigma \in \quotient{S_n}{S_k \times S_{n-k}}$ and $\tau_1\times \tau_2 \in S_k \times S_{n-k}$.
\end{rem}

The following lemma from \cite{MR1811255} bounds the operator norm of a homogeneous linear combination of creation or annihilation operators in terms of the $q$-Fock space norm of the comparable linear combination of basis words.  We will use it repeatedly to compensate for the extra terms in our expansions that arise because of the $q$-commutation relations instead of the freeness of the operators in the $q=0$ setting.

\begin{lm}[Theorem 2.2 in \cite{MR1811255}]\label{norm_of_creation}
\begin{align*}
    \left\|\sum_{\substack{w \in [d, \overline{d}],\\|w|=n}}\alpha_w a_w\right\|&\le C_{|q|}^{\frac{1}{2}}\left\|\sum_{\substack{w \in [d, \overline{d}],\\|w|=n}}\alpha_w e_w\right\|_{\mathcal{F}_q(H)} \\
    \left\|\sum_{\substack{w \in [d, \overline{d}],\\|w|=n}}\alpha_w a_w^*\right\|&\le C_{|q|}^{\frac{1}{2}}\left\|\sum_{\substack{w \in [d, \overline{d}],\\|w|=n}}\alpha_w e_w\right\|_{\mathcal{F}_q(H)} 
\end{align*}
\end{lm}

\begin{cor}\label{norm_of_prodcuts}
    Let $I$ be a finite set and $m \in \mathbb{N}$. If $\{u_i\}_{i\in I}, 
    \{v_i\}_{i \in I} \subset [d,\overline{d}]^*$ satisfy $|u_i|=|u_{i'}|$ and $|v_i|=|v_{i'}|$ for any $i,i'\in I$, then we have,
\[ \left\|\sum_{i\in I}\alpha_i a_{u_i}a^*_{v_i}\right\|\le C_{|q|}\left\| \sum_{i\in I}\alpha_i e_{u_i}\otimes e_{v_i} \right\|_{\mathcal{F}_q(H)^{\otimes 2}}.\]
\end{cor}
\begin{proof}
    The proof is based on the same argument as in the proof of Proposition 2.1 in Bo\.{z}ejko's paper \cite{MR1811255}. We take $\xi=\sum_{l}\xi_l$ in the algebraic Fock space with $\xi_l \in H^{\otimes l}$. Since $|u_i|=|u_{i'}|$ and $|v_i|=|v_{i'}|$ for any $i,i'\in I$, $\sum_{i\in I}\alpha_i a_{u_i}a^*_{v_i}\xi_l$ is orthogonal to $\sum_{i\in I}\alpha_i a_{u_i}a^*_{v_i}\xi_{l'}$ if $l \neq l'$. By using $P^{(l+|u_i|-|v_i|)}\le C_{|q|} P^{(|u_i|)}\otimes P^{(l-|v_i|)}$, we have 
    \begin{align*}
        \left\|\sum_{i\in I}\alpha_i a_{u_i}a^*_{v_i}\xi\right\|^2_{\mathcal{F}_q(H)} &= \sum_{l}\left\|\sum_{i\in I}\alpha_i a_{u_i}a^*_{v_i}\xi_l\right\|^2_{\mathcal{F}_q(H)}\\
        &\le C_{|q|}\sum_{l}\sum_{i_1,i_2\in I} \alpha_{i_1}\overline{\alpha_{i_2}}\langle e_{u_{i_1}}, e_{u_{i_2}}\rangle_q \left\langle a^*_{v_{i_1}}\xi_l,a^*_{v_{i_2}}\xi_l\right\rangle_q.
    \end{align*}
    We take the orthogonal decomposition of $\alpha_i e_{u_i}=\sum_{s}\beta(s,i)f_s$ for some orthonormal basis of $H^{\otimes |u_i|}$ with respect to the $q$-inner product. Then, we have
    \begin{align*}
        \left\|\sum_{i\in I}\alpha_i a_{u_i}a^*_{v_i}\xi\right\|^2_{\mathcal{F}_q(H)} &\le C_{|q|}\sum_{l}\sum_{i_1,i_2\in I} \alpha_{i_1}\overline{\alpha_{i_2}}\langle e_{u_{i_1}}, e_{u_{i_2}}\rangle_q \left\langle a^*_{v_{i_1}}\xi_l,a^*_{v_{i_2}}\xi_l\right\rangle_q\\
         &= C_{|q|}\sum_{l}\sum_{s}\sum_{i_1,i_2\in I} \beta(s,i_1)\overline{\beta(s,i_2)}\left\langle a^*_{v_{i_1}}\xi_l,a^*_{v_{i_2}}\xi_l\right\rangle_q\\
         &=C_{|q|}\sum_{l}\sum_{s}\left \|\sum_{i\in I} \beta(s,i) a_{v_i}^*\xi_l\right\|^2_{\mathcal{F}_q(H)}
         \end{align*}
Since $|v_i|=|v_{i'}|$ for any $i,i'\in I$, $\sum_{i\in I} \beta(s,i) a_{v_i}^*\xi_l$ is orthogonal to $\sum_{i\in I} \beta(s,i) a_{v_i}^*\xi_{l'}$ is $l\neq l'$. This implies $\sum_{l}\left \|\sum_{i\in I} \beta(s,i) a_{v_i}^*\xi_l\right\|^2_{\mathcal{F}_q(H)} =\left \|\sum_{i\in I} \beta(s,i) a_{v_i}^*\xi\right\|^2_{\mathcal{F}_q(H)}$. Then, by applying Lemma \ref{norm_of_creation}, we have
         \begin{align*}
         C_{|q|}\sum_{s}\left \|\sum_{i\in I} \beta(s,i) a_{v_i}^*\xi\right\|^2_{\mathcal{F}_q(H)}
         &\le C_{|q|}^{2}\sum_{s}\left \|\sum_{i\in I} \beta(s,i) e_{v_i}\right\|^2_{\mathcal{F}_q(H)} \|\xi\|_{\mathcal{F}_q(H)}^2 \\
         &=  C_{|q|}^{2} \left\| \sum_{i\in I}\alpha_i e_{u_i}\otimes e_{v_i}\right\|_{\mathcal{F}_q(H)^{\otimes 2}}^2 \|\xi\|_{\mathcal{F}_q(H)}^2  \end{align*}
     where we use the definition of $\beta(s,i)$ in the last inequality.
\end{proof}

We now introduce some notation and constructions that will aid in the readability of the following computations.

\begin{dfn} \label{definition.star} Define $\ast\colon H\to H$, $\xi\mapsto\xi^\ast$ to be the unique sesquilinear map satisfying $e_i^\ast = e_i$ for each basis vector $e_i$; i.e. $(\sum_i \alpha_i e_i)^\ast = \sum_i \bar\alpha_i e_i$.  We then extend $\ast$ to a sesquilinear map on $H^{\otimes n}$ for each $n\in\mathbb{N}$ by
\begin{equation} \label{eq:star.skew.linear} (\xi_1\otimes \xi_2\otimes  \cdots \otimes \xi_n)^\ast=\xi_n^\ast\otimes \xi_{n-1}^\ast\otimes \cdots \otimes \xi_1^\ast.  \end{equation}
For each $n \in \mathbb{N}$, we consider linear maps $a_{n}$ and $a^*_n$ from $H^{\otimes n}$ to $B(\mathcal{F}_q(H))$ defined by
\begin{equation}
\begin{aligned} \label{eq:an.an*}
    a_n(\xi_1\otimes \xi_2\otimes \cdots \otimes \xi_n)&=a(\xi_1)a(\xi_2)\cdots a(\xi_n),\\
    a^*_n(\xi_1\otimes \xi_2\otimes \cdots \otimes \xi_n)&=a^*(\xi_1^*)a^*(\xi_2^*)\cdots a^*(\xi_n^*)
\end{aligned}
\end{equation}

Let $M_n$ denote the multiplication operator which is a linear map from the algebraic tensor product $B(\mathcal{F}_q(H))^{\otimes n}$ to $B(\mathcal{F}_q(H))$ defined by
\[ M_n(a_1\otimes a_2 \otimes \cdots \otimes a_n)=a_1a_2\cdots a_n. \]
We omit the index $n$ of $M_n$ and simply write $M$.
\end{dfn}

\begin{rem} Comparing \eqref{eq:star.skew.linear} and \eqref{eq:an.an*}, the reader may find it counterintuitive that $a_n^\ast(\xi_1\otimes\cdots\otimes\xi_n) \ne [a_n(\xi_1\otimes\cdots\otimes\xi_n)]^\ast$.  Indeed, what holds with these conventions is
\begin{align*} [a_n(\xi_1\otimes\xi_2\otimes\cdots\otimes\xi_n)]^\ast = a^\ast(\xi_n)\cdots a^\ast(\xi_2) a^\ast(\xi_1) 
&= a_n^\ast (\xi_n^\ast \otimes\cdots\otimes\xi_2^\ast \otimes\xi_1^\ast) \\
&= a_n^\ast((\xi_1\otimes\xi_2\otimes\cdots\otimes\xi_n)^\ast).
\end{align*}
Thus, we have for $\psi\in H^{\otimes n}$,
\[[a_n(\psi)]^\ast = a_n^\ast(\psi^\ast). \]

(The relation $[a_n(\psi)]^\ast = a_n^\ast(\psi^\ast)$ for $\psi\in H^{\otimes n}$ holds even for $n=1$, directly from \eqref{eq:an.an*}.  In particular, it is important to note that the two maps $a(\cdot),a_1(\cdot)\colon H\to B(\mathcal{F}_q(H))$ are not equal, although the two do agree on ``real'' vectors, i.e.\ when $\xi=\xi^\ast$.)
This convention will be more convenient for the ordered products considered in all our computations to follow. \end{rem}

\begin{rem} We also use $\ast$ to denote the operation on words in $[d,\bar{d}]^\ast$ introduced on page \pageref{ast.on.words}.  The two uses are in fact consistent.  Since the unit vectors $e_i$ are (by definition) selfadjoint, $e_i^\ast = e_i$, it follows that
\[ e_{w}^\ast = (e_{w_1}\otimes\cdots\otimes e_{w_n})^\ast = e_{w_n}\otimes\cdots\otimes e_{w_1} = e_{w^\ast}. \]
\end{rem}

With the above notations, we can reformulate Lemma \ref{norm_of_prodcuts} as follows (a similar formulation also appears in \cite[Page 23]{MR2091676}).

\begin{cor}\label{generalproduct}
   For $\xi \in H^{\otimes n}\otimes H^{\otimes m}$, we have
       \[\|M(a_n\otimes a^*_m)\xi \| \le C_{|q|}\| \xi \|_{\mathcal{F}_q(H)^{\otimes 2}}\]
       where we embed $H^{\otimes n}\otimes H^{\otimes m}$ into $\mathcal{F}_q(H)^{\otimes 2} $.
\end{cor}

The following lemma gives the expansion of polynomials in the $q$-circular system with respect to the creation and annihilation operators, which is often useful to estimate operator norms.
\begin{lm}\label{qCircular}
For each $w=w_1\cdots w_n \in [d]^*$, we have
\begin{align*}
c_w^{(q)} &= \sum_{k=0}^n \sum_{\pi \in \quotient{S_n}{S_k \times S_{n-k}}} q^{\mathrm{inv}(\pi)} a_{\pi^{-1}(w)_{\le k}}a^*_{\overline{\pi^{-1}(w)_{>k}}}\\
&= \sum_{k=0}^n M(a_{k}\otimes a^*_{n-k})(I_{k}\otimes \overline{I}_{n-k})R^*_{k,n} e_w
\end{align*}
where $w_{\le k}=w_1\cdots w_k $ and $w_{>k}=w_{k+1}\cdots w_n$ for $k \le n$ and $\overline{I}_{n-k}$ is a linear map from $H_1^{\otimes n-k} $ to $H_2^{\otimes n-k}$ defined by $\overline{I}_{n-k} e_w=e_{\overline{w}}$.
As a consequence, we have
  \begin{align*}
  \sum_{|w|=n} \alpha_w c_w^{(q)} &= \sum_{k=0}^n \sum_{|w|=n} \sum_{\pi \in \quotient{S_n}{S_k \times S_{n-k}}} \alpha_{w}\ q^{\mathrm{inv}(\pi)} a_{\pi^{-1}(w)_{\le k}}a^*_{\overline{\pi^{-1}(w)_{>k}}}\\
  &=\sum_{k=0}^n M(a_k\otimes a^*_{n-k}) (I_k\otimes \overline{I}_{n-k})R^*_{k,n} \left(\sum_{|w|=n} \alpha_w e_w\right)
  \end{align*}
  and $\|\sum_{|w|=n} \alpha_w c_w^{(q)}\|_2= \|\sum_{|w|=n} \alpha_w e_w\|_{\mathcal{F}_q(H)}$.
\end{lm}

\begin{proof}
    We expand the product $c_w^{(q)}=\prod_{i=1}^n (a_{w_i}+a^*_{\overline{w_i}})$ to obtain the desired formula. The summation  index $k$ counts how many $a_{w_i}$ we pick from the product. When $k$ is fixed, each term $a_{w_1}^{\epsilon_1}\ldots a_{w_n}^{\epsilon_n}$ (where $\epsilon=\pm 1 $ and $a_{w_i}^{+1}= a_{w_i}$ and $a_{w_i}^{-1}=a^*_{\overline{w_i}}$) in the expansion of $c_w^{(q)}$ is associated with a permutation $\pi \in \quotient{S_n}{S_k \times S_{n-k}}$ in the following way; for each $1\le i \le k$, $\pi(i)$ is the $i$-th smallest number in $\{j \ | \ \epsilon_j=+1\}$, and for each $k+1\le i \le n$, $\pi(i)$ is the $(i-k)$-th smallest number in $\{j \ |\  \epsilon_j=-1\}$. Then we rearrange the product $a_{w_1}^{\epsilon_1}\ldots a_{w_n}^{\epsilon_n}$ in the form $a_{w_{i_1}}\cdots a_{w_{
    i_k}}a_{\overline{w_{i_{k+1}}}}^*\cdots a_{\overline{w_{i_{n}}}}^*$ by using the $q$-commutation relation:
    \[a_{\overline{i}}^*a_{j}=qa_{j}a_{\overline{i}}^*.\]
    By definition, we can see 
    \[a_{w_{i_1}}\cdots a_{w_{
    i_k}}a_{\overline{w_{i_{k+1}}}}^*\cdots a_{\overline{w_{i_{n}}}}^*=a_{\pi^{-1}(w)_{\le k}}a^*_{\overline{\pi^{-1}(w)_{>k}}}, \]
    and the number of $q$ appearing by this rearrangement is exactly $\mathrm{inv}(\pi)$. Therefore, we have
    \[a_{w_1}^{\epsilon_1}\ldots a_{w_n}^{\epsilon_n}=q^{\mathrm{inv}(\pi)} a_{\pi^{-1}(w)_{\le k}}a^*_{\overline{\pi^{-1}(w)_{>k}}}.\]
    Note that by the formula of $c_w^{(q)}$, we have
    \[c_w^{(q)}\Omega=e_w.\]
    Thus, we have $\|\sum_{|w|=n} \alpha_w c_w^{(q)}\|_2= \|\sum_{|w|=n} \alpha_w e_w\|_{\mathcal{F}_q(H)}$.
\end{proof}
\begin{rem}\label{selfadjointanalog}
The formula in Lemma \ref{qCircular} has a similar form to the formula for $q$-Wick polynomials. The $q$-Gausssian system is given by 
\[X_i^{(q)} = a_i + a_i^*, \quad i \in [d].\]
    There is an isomorphism $D$ between the $q$-Fock space $\mathcal{F}_q(\mathbb{C}^d)$ and GNS Hilbert space $L^2(\mathbb{C}\langle X_1^{(q)},\cdots,X_d^{(q)}\rangle, \tau)$ such that $D(T\Omega)=T$ for $T \in \mathbb{C}\langle X_1^{(q)},\ldots,X_d^{(q)}\rangle$. For $e_w \in \mathcal{F}_q(\mathbb{C}^d)$ ($w \in [d]^*$), $De_w$ is a polynomial in $X_1^{(q)},\ldots,X_d^{(q)}$ determined by the recursion, and its expansion in terms of the creation and annihilation operators (\cite[Proposition 1.1]{MR1811255}) is given by
    \[De_w(X_1^{(q)},\ldots,X_d^{(q)})=\sum_{k=0}^n \sum_{\pi \in \quotient{S_n}{S_k \times S_{n-k}}} q^{\mathrm{inv}(\pi)} a_{\pi^{-1}(w)_{\le k}}a^*_{\pi^{-1}(w)_{>k}}.\]
\end{rem}
The following lemma is important for the proof of strong Haagerup inequality.
\begin{lm}\label{AnnihilationCreation}
Let $u, v \in [d]^*$ with $|v|=k$ and $|u|=l$. Then we have,
\begin{align*}
(a_v)^*a_u=\sum_{m=0}^{\min(k,l)} q^{(k-m)(l-m)}&\sum_{\pi_1 \in \quotient{S_k}{S_{k-m} \times S_{m}}}\sum_{\pi_2 \in \quotient{S_l}{S_{m} \times S_{l-m}}}  q^{\mathrm{inv}(\pi_1)+\mathrm{inv}(\pi_2)}\\
&\cdot \langle  P^{(m)}e_{[\pi_1^{-1} (v^*)_{>k-m}]^*},e_{\pi_2^{-1}(u)_{\le m}} \rangle 
a_{\pi_2^{-1}(u)_{>m}} a_{\pi_1^{-1} (v^*)_{\le k-m}}^*.
\end{align*}
\end{lm}
\begin{rem}
    In the one-variable case ($d=1$), our formula implies
    \[a^{*k}a^l=\sum_{m=0}^{\min(k,l)} q^{(k-m)(l-m)}\binom{k}{k-m}_q\binom{l}{m}_q [m]_q! \ a^{l-m}a^{*k-m}.\]
\end{rem}

\begin{proof}
    We may assume $k \le l$ since if $l < k$, we have
\begin{align*}
(a_u)^*a_v=\sum_{m=0}^{l} q^{(k-m)(l-m)}&\sum_{\pi_2 \in \quotient{S_l}{S_{l-m} \times S_{m}}}  \sum_{\pi_1 \in \quotient{S_k}{S_{m} \times S_{k-m}}}q^{\mathrm{inv}(\pi_1)+\mathrm{inv}(\pi_2)}\\
&\cdot \langle  P^{(m)}e_{[\pi_2^{-1} (u^*)_{>l-m}]^*},e_{\pi_1^{-1}(u)_{\le m}} \rangle 
a_{\pi_1^{-1}(v)_{>m}} a_{\pi_2^{-1} (u^*)_{\le l-m}}^*
\end{align*}
    and by taking the adjoint, we obtain
\begin{align*}
(a_v)^*a_u =&\sum_{m=0}^{l} q^{(k-m)(l-m)}\sum_{\pi_2 \in \quotient{S_l}{S_{l-m} \times S_{m}}}  \sum_{\pi_1 \in \quotient{S_k}{S_{m} \times S_{k-m}}}q^{\mathrm{inv}(\pi_1)+\mathrm{inv}(\pi_2)}\\
&\cdot \langle  P^{(m)}e_{\pi_1^{-1}(v)_{\le m}},e_{[\pi_2^{-1} (u^*)_{>l-m}]^*}\rangle 
a_{[\pi_2^{-1} (u^*)_{\le l-m}]^*}a^*_{[\pi_1^{-1}(v)_{>m}]^*}\\
=&\sum_{m=0}^{l} q^{(k-m)(l-m)}\sum_{\pi_2^* \in \quotient{S_l}{S_{m} \times S_{l-m}}}  \sum_{\pi_1^* \in \quotient{S_k}{S_{k-m} \times S_{m}}}q^{\mathrm{inv}(\pi^*_1)+\mathrm{inv}(\pi_2^*)}\\
&\cdot \langle  P^{(m)}e_{[{\pi_1^*}^{-1}(v^*)_{>k- m}]^*},e_{{\pi_2^*}^{-1} (u)_{\le m}}\rangle 
a_{{\pi_2^*}^{-1} (u)_{>m}}a^*_{{\pi_1^*}^{-1}(v^*)_{\le k-m}}
\end{align*}
where $\quotient{S_l}{S_{l-m} \times S_{m}}\ni \pi \mapsto \pi^*\in \quotient{S_l}{S_{m} \times S_{l-m}}$ is given $\pi^*=w_0 \pi w_0$ with $w_0(i)=(l+1-i)$ for $1\le i \le l$ which preserves the number of inversions (see the proof of \cite[Theorem 2.2 (c)]{MR1811255}).

    We prove this lemma by strong induction on $0 \le k \le l$; the base case $k=0$ is the tautology $(a_v)^\ast a_u = (a_v)^\ast a_u$. Suppose that the formula holds for $k<l$. 
\begin{align*}
(a_{vi})^*a_u&=a_i^*(a_v)^*a_u\\
&=\sum_{m=0}^{k} q^{(k-m)(l-m)}\sum_{\pi_1 \in \quotient{S_k}{S_{k-m} \times S_{m}}}\sum_{\pi_2 \in \quotient{S_l}{S_{m} \times S_{l-m}}}  q^{\mathrm{inv}(\pi_1)+\mathrm{inv}(\pi_2)}\\
&\qquad\langle  P^{(m)}e_{[\pi_1^{-1} (v^*)_{>k-m}]^*},e_{\pi_2^{-1}(u)_{\le m}} \rangle 
a_i^* a_{\pi_2^{-1}(u)_{>m}} a_{\pi_1^{-1} (v^*)_{\le k-m}}^*
\end{align*}
We expand $a_i^* a_{\pi_2^{-1}(u)_{>m}}$ by iterating the $q$-commutation relation. Then, we have two cases; $a_i^*a_{\pi_2^{-1}(u)_{j}} $ for some $m+1 \le j \le l$; $a_i^*$ commutes with $a_{\pi_2^{-1}(u)_{j}}$ for all $m+1 \le j \le l$. Thus, the above sum is equal to 
\begin{align*}
&\sum_{m=0}^{k} q^{(k+1-m)(l-m)}\sum_{\pi_1 \in \quotient{S_k}{S_{k-m} \times S_{m}}}\sum_{\pi_2 \in \quotient{S_l}{S_{m} \times S_{l-m}}}  q^{\mathrm{inv}(\pi_1)+\mathrm{inv}(\pi_2)}\\
& \cdot\langle  P^{(m)}e_{[\pi_1^{-1} (v^*)_{>k-m}]^*},e_{\pi_2^{-1}(u)_{\le m}} \rangle 
 a_{\pi_2^{-1}(u)_{>m}} a_{i\pi_1^{-1} (v^*)_{\le k-m}}^*\\
 &+\sum_{m=0}^{k} q^{(k-m)(l-m)}\sum_{\pi_1 \in \quotient{S_k}{S_{k-m} \times S_{m}}}\sum_{\pi_2 \in \quotient{S_l}{S_{m} \times S_{l-m}}}  q^{\mathrm{inv}(\pi_1)+\mathrm{inv}(\pi_2)}\\
& \cdot \langle  P^{(m)}e_{[\pi_1^{-1} (v^*)_{>k-m}]^*},e_{\pi_2^{-1}(u)_{\le m}} \rangle 
 \sum_{j=m+1}^{l}q^{j-m-1}\delta_{i,\pi_2^{-1}(u)_j}a_{\pi_2^{-1}(u)_{(m,l]\setminus \{j\}}} a_{\pi_1^{-1} (v^*)_{\le k-m}}^*
 \end{align*}
 We also expand $\langle  P^{(m)}e_{[\pi_1^{-1} (v^*)_{>k-m}]^*},e_{\pi_2^{-1}(u)_{\le m}} \rangle$ by definition of $P^{(m)}$, and the quantity above is equal to
 \begin{align*}
 &\sum_{m=0}^{k} q^{(k+1-m)(l-m)}\sum_{\pi_1 \in \quotient{S_k}{S_{k-m} \times S_{m}}}\sum_{\pi_2 \in \quotient{S_l}{S_{m} \times S_{l-m}}}  q^{\mathrm{inv}(\pi_1)+\mathrm{inv}(\pi_2)}\\
& \cdot \langle  P^{(m)}e_{[\pi_1^{-1} (v^*)_{>k-m}]^*},e_{\pi_2^{-1}(u)_{\le m}} \rangle 
 a_{\pi_2^{-1}(u)_{>m}} a_{i\pi_1^{-1} (v^*)_{\le k-m}}^* \\
&+ \sum_{m=0}^{k} q^{(k-m)(l-m)}\sum_{\pi_1 \in \quotient{S_k}{S_{k-m} \times S_{m}}}\sum_{\pi_2 \in \quotient{S_l}{S_{m} \times S_{l-m}}} \sum_{\sigma \in S_m} \sum_{j=m+1}^{l} q^{j-m-1+ \mathrm{inv}(\pi_1)+\mathrm{inv}(\pi_2)+\mathrm{inv}(\sigma)}\\
& \cdot \langle  e_{\sigma([\pi_1^{-1} (v^*)_{>k-m}]^*)},e_{\pi_2^{-1}(u)_{\le m}} \rangle 
 \delta_{i,\pi_2^{-1}(u)_j}a_{\pi_2^{-1}(u)_{(m,l]\setminus \{j\}}} a_{\pi_1^{-1} (v^*)_{\le k-m}}^*
\end{align*}
where $a_{\pi_2^{-1}(u)_{(m,l]\setminus \{j\}}}=a_{\pi_2^{-1}(u)_{m+1}}\cdots a_{\pi_2^{-1}(u)_{j-1}}a_{\pi_2^{-1}(u)_{j+1}}\cdots a_{\pi_2^{-1}(u)_{l}}$.
    For the first sum,  we replace $\pi_1 \in \quotient{S_k}{S_{k-m} \times S_{m}}$ with $\pi'_1 \in \quotient{S_{k+1}}{S_{k+1-m} \times S_{m}}$, $\pi'_1(1)=1$ defined by $\pi'_1(1)=1$ and $\pi'_1(j)=\pi(j-1)+1$ for $j>1$. Note that we have 
    \[\mathrm{inv}(\pi'_1)=\mathrm{inv}(\pi_1).\]
    For the second sum, we associate $\pi_1 \in \quotient{S_k}{S_{k-m} \times S_{m}}$, $\pi_2 \in \quotient{S_l}{S_{m} \times S_{l-m}}$, $\sigma \in S_m$, and $m+1\le j \le l$ with $\pi'_1 \in \quotient{S_{k+1}}{S_{k-m} \times S_{m+1}} $ such that $\pi'_1(k-m+1)=1$, $\pi_2' \in \quotient{S_l}{S_{m+1} \times S_{l-m-1}} $, $\sigma' \in S_{m+1}$ as follows; $\pi'_1(k-m+1)=1$ and $\pi'_1(j)|_{[1,m+1]\setminus \{k-m+1\}}=\pi_1$;  $\pi_2'(s)=\pi_2(j)$ for $1\le s\le m+1$ such that $\pi_2(s-1)<\pi_2(j)<\pi_2(s)$ (when $\pi_2(m)<\pi_2(j)$, we set $s=m+1$), $\pi'_2|_{[1,l]\setminus \{s\}}=\pi_2|_{[1,l]\setminus \{j\}}$; $\sigma'(m+1)=s$ and $\sigma'|_{[1,m]}=\sigma$ in a order-preserving way. Note that we have
    \begin{align*} \mathrm{inv}(\pi'_1)&=\mathrm{inv}(\pi_1)+k-m,\\
    \mathrm{inv}(\pi'_2)&=\mathrm{inv}(\pi_2)+(j-m-1)-(m+1-s),\\
    \mathrm{inv}(\sigma')&=\mathrm{inv}(\sigma)+(m+1-s),
    \end{align*}
and 
\begin{align*} \langle  &e_{\sigma'([{\pi'}_1^{-1} (iv^*)_{>k-m}]^*)},e_{{\pi'}_2^{-1}(u)_{\le m+1}} \rangle  a_{{\pi'}_2^{-1}(u)_{>m+1}} a_{{\pi'}_1^{-1} (iv^*)_{\le k-m}}^* \\
&=\langle  e_{\sigma([\pi_1^{-1} (v^*)_{>k-m}]^*)},e_{\pi_2^{-1}(u)_{\le m}} \rangle 
 \delta_{i,\pi_2^{-1}(u)_j}a_{\pi_2^{-1}(u)_{(m,l]\setminus \{j\}}} a_{\pi_1^{-1} (v^*)_{\le k-m}}^*.
 \end{align*}
 Since this correspondence is one-to-one, we rewrite the sum by using $\pi'_1,\pi'_2,\sigma'$ instead of $\pi_1,\pi_2,\sigma$, and we have 
 
 \begin{align*}
&\sum_{m=0}^{k} q^{(k+1-m)(l-m)}\sum_{\substack{\pi'_1 \in \quotient{S_{k+1}}{S_{k+1-m} \times S_{m}}\\ \pi'_1(1)=1}}\sum_{\pi_2 \in \quotient{S_l}{S_{m} \times S_{l-m}}}  q^{\mathrm{inv}(\pi'_1)+\mathrm{inv}(\pi_2)}\\
& \cdot \langle  P^{(m)}e_{[{\pi'}_1^{-1} (iv^*)_{>k+1-m}]^*},e_{\pi_2^{-1}(u)_{\le m}} \rangle 
 a_{\pi_2^{-1}(u)_{>m}} a_{{\pi'}_1^{-1} (iv^*)_{\le k+1-m}}^* \\
&+ \sum_{m=0}^{k} q^{(k-m)(l-m-1)}\sum_{\substack{\pi'_1 \in \quotient{S_{k+1}}{S_{k-m} \times S_{m+1}}\\ \pi'_1(k-m+1)=1}}\sum_{\pi'_2 \in \quotient{S_l}{S_{m+1} \times S_{l-m-1}}} \sum_{\sigma' \in S_{m+1}} q^{ \mathrm{inv}(\pi'_1)+\mathrm{inv}(\pi'_2)+\mathrm{inv}(\sigma')}\\
& \cdot\langle  e_{\sigma'([{\pi'}_1^{-1} (iv^*)_{>k-m}]^*)},e_{{\pi'}_2^{-1}(u)_{\le m+1}} \rangle 
 a_{{\pi'}_2^{-1}(u)_{>m+1}} a_{{\pi'}_1^{-1} (iv^*)_{\le k-m}}^*
 \end{align*}
 By replacing $m+1$ with $m$ in the second sum, we have
 \begin{align*}
 &\sum_{m=0}^{k} q^{(k+1-m)(l-m)}\sum_{\substack{\pi'_1 \in \quotient{S_{k+1}}{S_{k+1-m} \times S_{m}}\\ \pi'_1(1)=1}}\sum_{\pi_2 \in \quotient{S_l}{S_{m} \times S_{l-m}}}  q^{\mathrm{inv}(\pi'_1)+\mathrm{inv}(\pi_2)}\\
& \cdot\langle  P^{(m)}e_{[{\pi'}_1^{-1} (iv^*)_{>k+1-m}]^*},e_{\pi_2^{-1}(u)_{\le m}} \rangle 
 a_{\pi_2^{-1}(u)_{>m}} a_{{\pi'}_1^{-1} (iv^*)_{\le k+1-m}}^* \\
&+ \sum_{m=1}^{k+1} q^{(k+1-m)(l-m)}\sum_{\substack{\pi'_1 \in \quotient{S_{k+1}}{S_{k+1-m} \times S_{m}}\\ \pi'_1(k+1-m+1)=1}}\sum_{\pi'_2 \in \quotient{S_l}{S_{m} \times S_{l-m}}} q^{ \mathrm{inv}(\pi'_1)+\mathrm{inv}(\pi'_2)}\\
& \cdot \langle P^{(m)} e_{[{\pi'}_1^{-1} (iv^*)_{>k+1-m}]^*},e_{{\pi'}_2^{-1}(u)_{\le m}} \rangle 
 a_{{\pi'}_2^{-1}(u)_{>m}} a_{{\pi'}_1^{-1} (iv^*)_{\le k+1-m}}^*.
 \end{align*}
 Note that $\pi'_1 \in \quotient{S_{k+1}}{S_{k+1-m} \times S_{m}}$ satisfies either ${\pi'}_1^{-1}(1)=1$ or ${\pi'}_1^{-1}(1)=k+1-m+1$, and ${\pi'}_1^{-1}(1)=1$ when $m=0,k+1$. Therefore, the above sum is equal to
 
 \begin{align*}
 &\sum_{m=0}^{k+1} q^{(k+1-m)(l-m)}\sum_{\pi'_1 \in \quotient{S_{k+1}}{S_{k+1-m} \times S_{m}}}\sum_{\pi'_2 \in \quotient{S_l}{S_{m} \times S_{l-m}}} q^{ \mathrm{inv}(\pi'_1)+\mathrm{inv}(\pi'_2)}\\
& \cdot \langle P^{(m)} e_{[{\pi'}_1^{-1} (iv^*)_{>k+1-m}]^*},e_{{\pi'}_2^{-1}(u)_{\le m}} \rangle 
 a_{{\pi'}_2^{-1}(u)_{>m}} a_{{\pi'}_1^{-1} (iv^*)_{\le k+1-m}}^*,
 \end{align*}
 which completes the induction.
\end{proof}
We can reformulate this lemma by using $a_n$ and $a_n^*$, cf.\ \eqref{eq:an.an*}.  Note that the operation $\ast$ from Definition \ref{definition.star} is anti-unitary \cite[Theorem 2.2 (c)]{MR1811255}) and we have
\[ \langle  P^{(m)}e_{[\pi_1^{-1} (v^*)_{>k-m}]^*},e_{\pi_2^{-1}(u)_{\le m}} \rangle=\langle e_{\pi_1^{-1} (v^*)_{>k-m}},(e_{\pi_2^{-1}(u)_{\le m}})^* \rangle_q.\]
Since $(a_v)^*=a^*_k(e_{v^*})$, the left hand side of the formula in Lemma \ref{AnnihilationCreation} is equal to  $M(a^*_k\otimes a_l)(e_{v^*}\otimes e_u)$ and the right hand side can be written as
\[\sum_{m=0}^{\min(k,l)}q^{(k-m)(l-m)}M(a_{l-m}\otimes a^*_{k-m})U_{\mathrm{flip}}(I\otimes \Phi_m \otimes I)(R^*_{k-m,k}\otimes R^*_{m,l})(e_{v^*}\otimes e_u)\]
where $\Phi_m: H^{\otimes m} \otimes H^{\otimes m} \to \mathbb{C}$ is a linear map defined by 
\[\Phi_m(\xi\otimes \eta) = \langle \xi,\eta^*\rangle_q,\]
and $U_{\mathrm{flip}}$ is a linear map defined by the flip on $H^{\otimes k-m}\otimes H^{\otimes l-m}$, i.e. $U_{\mathrm{flip}}(a\otimes b)=b\otimes a$. In our notation, $I\otimes \Phi_m \otimes I$ acts on $H^{\otimes k-m}\otimes H^{\otimes m}\otimes H^{\otimes m}\otimes H^{\otimes l-m}$ where we apply $\Phi_m$ to the second and third tensor components and the identity operator $I$ to other tensor components. Since all operations are linear maps on $H^{\otimes k}\otimes H^{\otimes l}$ (which is finite-dimensional), we have the following
\begin{cor}[cf. Lemma 4.6 in \cite{MR4251282} ]\label{creationannihilation2}
As a linear map from $H^{\otimes k}\otimes H^{\otimes l}$ to $B(\mathcal{F}_q(H))$, we have the following identity, 
\[
M(a^*_k\otimes a_l)
=\sum_{m=0}^{\min(k,l)}q^{(k-m)(l-m)}M(a_{l-m}\otimes a^*_{k-m})U_{\mathrm{flip}}(I\otimes \Phi_m \otimes I)(R^*_{k-m,k}\otimes R^*_{m,l}). 
\]

\end{cor}
The next two lemmas are fundamental to our approach to the strong Haagerup inequality; we use them several times in Section \ref{stronghaag}.
\begin{lm}\label{Rstar}
   Let $k\le n$ be positive integers. For $\xi \in H^{\otimes n}$, we have
    \[\|R^*_{k,n}\xi\|_{\mathcal{F}_q(H)^{\otimes 2}} \le C_{|q|}^{\frac{1}{2}}\|\xi\|_{\mathcal{F}_q(H)}\]
where $R^*_{k,n}\xi\in H^{\otimes k} \otimes H^{\otimes n-k}$ is embedded into $\mathcal{F}_q(H)^{\otimes 2}$. 
\end{lm}
\begin{proof}
    This proof is based on the same argument in the proof of Proposition 2.1 in \cite{MR1811255}. By Lemma \ref{NCbinom}, we have $P^{(n)}=(P^{(k)}\otimes P^{(n-k)})R^*_{k,n}$. We apply this to $\|R^*_{k,n}\xi\|_{\mathcal{F}_q(H)^{\otimes 2}}^2$, and we get
    \begin{align*}
\|R^*_{k,n}\xi\|_{\mathcal{F}_q(H)^{\otimes 2}}^2  &=\langle (P^{(k)}\otimes P^{(n-k)})R^*_{k,n} \xi, R^*_{k,n} \xi \rangle\\
&=\langle P^{(n)}\xi,R^*_{k,n} \xi\rangle.
    \end{align*}
    Using the Cauchy-Schwarz inequality we obtain
    \[\langle P^{(n)}\xi,R^*_{k,n} \xi\rangle \le \|\xi\|_{\mathcal{F}_q(H)}\|R^*_{k,n} \xi\|_{\mathcal{F}_q(H)}. \]
    By Lemma \ref{NCbinom}, we also have $P^{(n)}\le C_{|q|}P^{(k)}\otimes P^{(n-k)}$. We apply this to $\|R^*_{k,n} \xi\|_{\mathcal{F}_q(H)}^2$ and we get
    \begin{align*}
        \|R^*_{k,n} \xi\|_{\mathcal{F}_q(H)}^2&=\langle P^{(n)}R^*_{k,n}\xi,R^*_{k,n} \xi\rangle\\
        &\le C_{|q|}\langle P^{(k)}\otimes P^{(n-k)} R^*_{k,n}\xi,R^*_{k,n} \xi\rangle\\
        &= C_{|q|}\langle P^{(n)}\xi,R^*_{k,n} \xi\rangle.
    \end{align*}
    Again, by the Cauchy-Schwarz inequality, we have
    \begin{align*}
        \|R^*_{k,n} \xi\|_{\mathcal{F}_q(H)}^2
        &\le C_{|q|}\langle P^{(n)}\xi,R^*_{k,n} \xi\rangle\\
        &\le C_{|q|}\|\xi\|_{\mathcal{F}_q(H)}
        \|R^*_{k,n} \xi\|_{\mathcal{F}_q(H)}.
    \end{align*}
    Therefore, we obtain
    \[ \|R^*_{k,n} \xi\|_{\mathcal{F}_q(H)}\le C_{|q|}\|\xi\|_{\mathcal{F}_q(H)},\]
    and
    \[\|R^*_{k,n}\xi\|_{\mathcal{F}_q(H)^{\otimes 2}}^2\le \|\xi\|_{\mathcal{F}_q(H)}\|R^*_{k,n} \xi\|_{\mathcal{F}_q(H)}\le C_{|q|}\|\xi\|_{\mathcal{F}_q(H)}^2.  \]
Taking square roots concludes the proof.
\end{proof}

\begin{lm}\label{tensorineq}
        Let $K,L,M$ be (complex) Hilbert spaces and $U:L\to L$ be an anti-unitary. Define a linear map $\Phi_U:L\otimes L \to \mathbb{C}$ by \[\Phi_U(\xi \otimes \eta)=\langle \xi,U\eta\rangle_{L}.\]
   Then, we have for $\xi \in K\otimes L$ and $\eta \in L\otimes M$
   \[\|(I_K\otimes \Phi_U \otimes I_{M} )(\xi \otimes \eta)\|_{K \otimes M}\le \|\xi\|_{K\otimes L}\|\eta\|_{L\otimes M}\]
   \end{lm}
   \begin{proof}
      Let $\{x_k\}\subset K$, $\{y_l\}\subset L$, $\{z_m\}\subset M$ be orthonormal bases. For $\xi=\sum_{k,l} \alpha_{k,l}x_k \otimes y_l $ and $\eta=\sum_{l,m}\beta_{l,m}y_l\otimes z_m$, we have
      \begin{align*}
       \|(I_K\otimes \Phi_U \otimes I_{M} )(\xi \otimes \eta)\|^2_{K\otimes M}&=\left\|\sum_{k,l,l',m}\alpha_{k,l}\beta_{l',m}\langle y_l,Uy_{l'}\rangle_L x_k \otimes z_m\right\|^2_{K \otimes M}\\
&=\sum_{k,m}\left|\sum_{l,l'}\alpha_{k,l}\beta_{l',m} \langle y_l,Uy_{l'}\rangle_L\right|^2\\
&=\sum_{k,m}\left| \left\langle \sum_{l}\alpha_{k,l}y_l,U\sum_{l'}\beta_{l',m}y_{l'}\right\rangle_L\right|^2.
       \end{align*}
       Using the Cauchy-Schwarz inequality, we have
       \begin{align*}
         \sum_{k,m}\left| \left\langle \sum_{l}\alpha_{k,l}y_l,U\sum_{l'}\beta_{l',m}y_{l'}\right\rangle_L\right|^2&\le \sum_{k,m} \left\|\sum_{l}\alpha_{k,l}y_l\right\|^2_L \left\|U\sum_{l}\beta_{l,m}y_{l}\right\|^2_L \\
         &=\sum_{k,m} \left\|\sum_{l}\alpha_{k,l}y_l\right\|^2_L \left\|\sum_{l}\beta_{l,m}y_{l}\right\|^2_L \\
         &=\left(\sum_{k,l}|\alpha_{k,l}|^2\right)\left(\sum_{m,l}|\beta_{l,m}|^2\right)\\
         &=\|\xi\|^2_{K\otimes L}\|\eta\|^2_{L\otimes M}
       \end{align*}
       where we use the assumption that $U$ is anti-unitary in the first equality.
       This concludes the proof.
   \end{proof}

\section{The strong Haagerup inequality for $q$-circular systems}\label{stronghaag}
To prove our main result, we prove the following lemma.
\begin{lm}\label{KeyLem}
     There exists $A=A(|q|)>0$ such that, for any $n \in \mathbb{N}$ and $\xi_k \in H_1^{\otimes k}\otimes H_2^{\otimes n-k}$ we have
       \[\left\|\sum_{k=1}^{n}M (a_k\otimes a^*_{n-k})\xi_k\right\|\le A\sqrt{n}\max_{1\le k\le n}\|\xi_k\|_{\mathcal{F}_q(H)^{\otimes 2}} \]
 \end{lm}
 To prove this Lemma, we estimate 5 operators obtained by the triangle inequality and Corollary  \ref{creationannihilation2}. First of all, we have (see Section \ref{freecase})
\begin{align*}
    \left\| \sum_{k=1}^{n}M (a_k\otimes a^*_{n-k})\xi_k\right\|^2
    &\le \sum_{k=1}^n \|A_k^*A_k\| + 2\left\|\sum_{k=1}^{n-1}\sum_{l=k+1}^n A_k^* A_l \right\| \\
    &=  \sum_{k=1}^n \|A_k\|^2 + 2\left\|\sum_{k=1}^{n-1}\sum_{l=k+1}^n A_k^* A_l \right\|
\end{align*}
where we set 
\[A_k = \sum_{k=1}^{n}M (a_k\otimes a^*_{n-k})\xi_k.\]
Subsequently, by applying Corollary \ref{creationannihilation2} and  
\begin{align*}
A_k^*A_l&=
   (M (a_k\otimes a^*_{n-k})\xi_k)^*M (a_l\otimes a^*_{n-l})\xi_l\\
   &=M (a_{n-k}\otimes a^*_{k}\otimes a_l\otimes a^*_{n-l}) (\xi_k^* \otimes \xi_l)\\
   &=\sum_{m=0}^kq^{(k-m)(l-m)}M (a_{n-k}\otimes  a_{l-m} \otimes a^*_{k-m} \otimes a^*_{n-l}) U_{\mathrm{(2,3)}}\Xi_{k,l,m} 
\end{align*}
where $\Xi_{k,l,m}\in H_2^{\otimes n-k}\otimes H_1^{\otimes k-m}\otimes H_1^{\otimes l-m}\otimes H_2^{\otimes n-l}$ is defined by
\[\Xi_{k,l,m}=(I\otimes \Phi_m \otimes I)(I\otimes R^*_{k-m,k}\otimes R^*_{m,l}\otimes I)(\xi_k^*\otimes \xi_l),   \]
and $U_{(2,3)}=I_{n-k}\otimes U_{\mathrm{flip}} \otimes I_{n-l}$ is the flip of the second and third components of $H_2^{\otimes n-k}\otimes H_1^{\otimes k-m}\otimes H_1^{\otimes l-m}\otimes H_2^{\otimes n-l}$. 
We decomposing this sum into two parts corresponding to $ 0\le m \le k-1$ and $m=k$, namely we have
\[A_k^*A_l=\sum_{m=0}^{k-1} A_{k,l,m} + A_{k,l,k}\]
where for $0\le m \le k-1$ 
\[A_{k,l,m}=q^{(k-m)(l-m)}M (a_{n-k}\otimes  a_{l-m} \otimes a^*_{k-m} \otimes a^*_{n-l}) U_{\mathrm{(2,3)}}\Xi_{k,l,m}\]
and for $m=k$
\begin{align*}
  A_{k,l,k}&=M(a_{n+l-2k}\otimes a^*_{n-l})(I_{n-k}\otimes \Phi_k \otimes I_{n-k})(I_{n}\otimes R^*_{k,l}\otimes I_{n-l})(\xi_k^*\otimes \xi_l)\\
  &=M(a_{n-k}\otimes a_{l-k}\otimes a^*_{n-l})(\Xi_{k,l})
  \end{align*}
  with
  \[\Xi_{k,l}=(I_{n-k}\otimes \Phi_k \otimes I_{n-k})(I_{n}\otimes R^*_{k,l}\otimes I_{n-l})(\xi_k^*\otimes \xi_l) \in H_2^{\otimes n-k} \otimes H_1^{\otimes l-k}\otimes H_2^{\otimes n-l}.\]
Thus by the triangle inequality, we have
\[
    \left\|\sum_{k=1}^{n-1}\sum_{l=k+1}^n A_k^* A_l \right\|\le \left\|\sum_{k=1}^{n-1}\sum_{l=k+1}^n\sum_{m=0}^{k-1}A_{k,l,m}\right\|+\left\|\sum_{k=1}^{n-1}\sum_{l=k+1}^nA_{k,l,k}\right\|
\]
As the next step, we estimate $ \left\|\sum_{k=1}^{n-1}\sum_{l=k+1}^nA_{k,l,k}\right\|^2$ in a similar way to the first step, yielding
\begin{align*}
    \left\|\sum_{k=1}^{n-1}\sum_{l=k+1}^nA_{k,l,k}\right\|^2&\le \sum_{k=1}^{n-1}\|B_k^*B_k\|+2\left\|\sum_{k_1=1}^{n-2}\sum_{k_2=k_1+1}^{n-1}B_{k_1}^*B_{k_2}\right\|
\end{align*}
where we put $B_k = \sum_{l=k+1}^n A_{k,l,k}$.
Finally, we expand $B_k^*B_k$ as follows
\[
      B_{k}^*B_{k}
      =\sum_{l_1,l_2}M(a_{n-l_1}\otimes a^*_{l_1-k}\otimes a^*_{n-k}\otimes  a_{n-k}\otimes a_{l_2-k}\otimes a^*_{n-l_2})(\Xi^*_{k,l_1}\otimes \Xi_{k,l_2}).
\]
      Then, we apply Corollary \ref{creationannihilation2} to the product of $ a^*_{n-k}$ and $a_{n-k}$, and $B_k^*B_k$ is equal to
\[\sum_{l_1,l_2}\sum_{m=0}^{n-k}q^{(n-k-m)^2} M(a_{n-l_1}\otimes a^*_{l_1-k}\otimes a_{n-k-m} \otimes a^*_{n-k-m} \otimes a_{l_2-k}\otimes a^*_{n-l_2}) (\Xi_{m})\]
  
  where $\Xi_m$ is 
  \[U_{(3,4)}(I \otimes \Phi_m \otimes I)(I\otimes R^*_{n-k-m,n-k}\otimes R^*_{m,n-k} \otimes I)(\Xi^*_{k,l_1}\otimes \Xi_{k,l_2}) \]
  with $U_{(3,4)}$ acting as the flip of third and fourth tensor components of $ H_2^{\otimes n-l_1}\otimes H_1^{\otimes l_1-k} \otimes H_2^{\otimes n-k-m} \otimes H_2^{\otimes n-k-m}\otimes H_1^{\otimes l_2-k} \otimes H_2^{\otimes n-l_2}$. Note that our notation says $(I\otimes \Phi_m \otimes I)$ acts on \[ H_2^{\otimes n-l_1}\otimes H_1^{\otimes l_1-k} \otimes H_2^{\otimes n-k-m} \otimes H_2^{\otimes m} \otimes H_2^{\otimes m} \otimes H_2^{\otimes n-k-m}\otimes H_1^{\otimes l_2-k} \otimes H_2^{\otimes n-l_2}\] 
  where we apply $\Phi_m$ to the fourth and fifth tensor components and the identity operator to other tensor components. 
  We decompose the sum into the two parts corresponding to $0\le m \le n-k-1$ and $m=n-k$, namely we have 
\[B_k^*B_k=\sum_{m=0}^{n-k-1}B_{k,m}+B_{k,n-k}.\]
Thus, by the triangle inequality, we have
\[\|B_k^*B_k\|\le \sum_{m=0}^{n-k-1}\|B_{k,m}\|+\|B_{k,n-k}\|. \]
According to this decomposition, we need to estimate the following 5 terms:
\begin{gather*}
    \sum_{k=1}^n \|A_k\|^2, \quad \left\|\sum_{k=1}^{n-1}\sum_{l=k+1}^n\sum_{m=0}^{k-1}A_{k,l,m}\right\|,\quad  \left\|\sum_{k_1=1}^{n-2}\sum_{k_2=k_1+1}^{n-1}B_{k_1}^*B_{k_2}\right\|, \\
   \quad  \sum_{k=1}^{n-1}\sum_{m=0}^{n-k-1}\|B_{k,m}\|, \quad \sum_{k=1}^{n-1}\|B_{k,n-k}\|.
\end{gather*}
In the following lemmas, we give a bound for each term and prove Lemma \ref{KeyLem} by combining them. 
\begin{lm}\label{step1}
    \[\sum_{k=1}^n \|A_k\|^2 \le C_{|q|}^2 n \max_{1\le k \le n} \left\| \xi_k \right\|_{\mathcal{F}_q(H)^{\otimes 2}}^2.\]
\end{lm}
\begin{proof}
    By Corollary \ref{generalproduct}, we have
    \[\sum_{k=1}^n \|A_k\|^2 \le n \max_{1\le k \le n}\|M(a_k\otimes a^*_{n-k})\xi_k\|^2 \le C_{|q|}^2 n \max_{1\le k\le n}\left\| \xi_k \right\|_{\mathcal{F}_q(H)^{\otimes 2}}^2.\]
\end{proof}
\begin{lm}\label{step2}
There exists $D_1=D_1(|q|)>0$ such that 
    \[\left\|\sum_{k=1}^{n-1}\sum_{l=k+1}^n\sum_{m=0}^{k-1}A_{k,l,m}\right\|\le D_1 n \max_{1 \le k \le n} \left\| \xi_k\right\|_{\mathcal{F}_q(H)^{\otimes 2}}^2\]
\end{lm}
\begin{proof}
   We replace $l-k$ with $l$ and rearrange the summands. Then we have
   \begin{align*}
       \left\|\sum_{k=1}^{n-1}\sum_{l=k+1}^n\sum_{m=0}^{k-1}A_{k,l,m}\right\|&=\left\|\sum_{m=0}^{n-2}\sum_{l=1}^{n-m-1}\sum_{k=m+1}^{n-l}A_{k,l+k,m}\right\|\\
       &\le \sum_{m=0}^{n-2}\sum_{l=1}^{n-m-1}\left\|\sum_{k=m+1}^{n-l}A_{k,l+k,m}\right\|.
   \end{align*}
Note that $\sum_{k=m+1}^{n-l} A_{k,l+k,m}$ is equal to
\begin{align*}
&\sum_{k=m+1}^{n-l}q^{(k-m)(l+k-m)}M (a_{n-k}\otimes  a_{l+k-m} \otimes a^*_{k-m} \otimes a^*_{n-l-k}) U_{\mathrm{(2,3)}}\Xi_{k,l+k,m}\\ 
&=M (a_{n+l-m} \otimes a^*_{n-l-m})\sum_{k=m+1}^{n-l}q^{(k-m)(l+k-m)}U_{\mathrm{(2,3)}} \Xi_{k,l+k,m}.
\end{align*}
Thus, by applying Corollary \ref{generalproduct}, we have
\[\left\|\sum_{m=0}^{k-1}A_{k,l,m}\right\|\le C_{|q|}\left\|\sum_{k=m+1}^{n-l}q^{(k-m)(l+k-m)}U_{\mathrm{(2,3)}}\Xi_{k,l+k,m}\right\|_{\mathcal{F}_q(H)^{\otimes 2}}. \]
Note that $U_{\mathrm{(2,3)}}\Xi_{k,l+k,m} \in  H_2^{\otimes n-k}\otimes  H_1^{\otimes l+k-m}\otimes  H_1^{\otimes k-m}\otimes H_2^{\otimes n-l-k}$ and we consider the norm on the tensor product of $H_2^{\otimes n-k}\otimes  H_1^{\otimes l+ k-m}\subset H^{\otimes n+l-m}$ ($\subset \mathcal{F}_q(H)$) and $ H_1^{\otimes k-m}\otimes H_2^{\otimes n-l-k}\subset H^{\otimes n-l-m}$ ($\subset \mathcal{F}_q(H)$). For fixed $l,m$, the number of tensor factors of $H_1$ and $H_2$ in $H_2^{\otimes n-k}\otimes  H_1^{\otimes l+k-m}$ (similarly in $H_1^{\otimes k-m}\otimes H_2^{\otimes n-l-k}$) differ as $k$ changes. Since $H_1$ is orthogonal to $H_2$, we have for $k\neq k'$, 
\begin{align*}
H_2^{\otimes n-k}\otimes  H_1^{\otimes l+k-m}&\perp H_2^{\otimes n-k'}\otimes  H_1^{\otimes l+ k'-m}\\ \ H_1^{\otimes k-m}\otimes H_2^{\otimes n-l-k}&\perp H_1^{\otimes k'-m}\otimes H_2^{\otimes n-l-k'}
\end{align*}
with respect to the $q$-inner product.
The orthogonality of $H_1$ and $H_2$ also tells that for $\xi,\xi' \in H_2^{\otimes n-k}\otimes  H_1^{\otimes l+k-m}$ and $\eta,\eta' \in H_1^{\otimes k-m}\otimes H_2^{\otimes n-l-k}$,
\begin{align*}
  \langle P^{(n+l-m)} \xi ,\xi' \rangle &=\langle P^{(n-k)}\otimes P^{(l+k-m)}\xi,\xi'\rangle \\
  \langle P^{(n-l-m)} \eta ,\eta' \rangle &=\langle P^{(k-m)}\otimes P^{(n-l-k)}\eta,\eta'\rangle
\end{align*}
Thus, we get
\begin{align*}&\left\|\sum_{k=m+1}^{n-l}q^{(k-m)(l+k-m)}U_{\mathrm{(2,3)}}\Xi_{k,l+k,m}\right\|_{\mathcal{F}_q(H)^{\otimes 2}}\\
&=\left(\sum_{k=m+1}^{n-l}|q|^{2(k-m)(l+k-m)}\left\|U_{\mathrm{(2,3)}}\Xi_{k,l+k,m}\right\|^2_{\mathcal{F}_q(H)^{\otimes 4}}\right)^{\frac{1}{2}}
\end{align*}
Since $U_{(2,3)}$ is the transposition of the second and third tensor components in $\mathcal{F}_q(H)^{\otimes 4}$, $U_{(2,3)}$ is unitary on $\mathcal{F}_q(H)^{\otimes 4}$ and thus

\[ \left\|U_{\mathrm{(2,3)}}\Xi_{k,l+k,m}\right\|^2_{\mathcal{F}_q(H)^{\otimes 4}}=\left\|\Xi_{k,l+k,m}\right\|^2_{\mathcal{F}_q(H)^{\otimes 4}}. \]

Note that we can write $\Xi_{k,l+k,m}$ as
\begin{align*}
\Xi_{k,l+k,m}&=(I\otimes \Phi_m \otimes I)(I\otimes R^*_{k-m,k}\otimes R^*_{m,l+k}\otimes I)(\xi_k^*\otimes \xi_{l+k})\\
&=(I\otimes \Phi_m \otimes I)\left[(I\otimes R^*_{k-m,k})\xi^*_k\otimes (R^*_{m,l+k}\otimes I)\xi_{l+k} \right].
\end{align*}
By Lemma \ref{tensorineq}, we have
\[ \left\|\Xi_{k,l+k,m}\right\|^2_{\mathcal{F}_q(H)^{\otimes 4}} \le \|(I\otimes R^*_{k-m,k})\xi^*_k\|^2_{\mathcal{F}_q(H)^{\otimes 3}} \|(R^*_{m,l+k}\otimes I)\xi_{l+k}\|^2_{\mathcal{F}_q(H)^{\otimes 3}}. \]
By using Lemma \ref{Rstar} and the fact $\ast$ is anti-unitary, we have
\begin{align*}
  \|(I\otimes R^*_{k-m,k})\xi^*_k\|^2_{\mathcal{F}_q(H)^{\otimes 3}}&\le C_{|q|}  \|\xi^*_k\|^2_{\mathcal{F}_q(H)^{\otimes 2}}=  C_{|q|}  \|\xi_k\|^2_{\mathcal{F}_q(H)^{\otimes 2}}\\
  \|(R^*_{m,l+k}\otimes I)\xi_{l+k}\|^2_{\mathcal{F}_q(H)^{\otimes 3}}&\le C_{|q|}\|\xi_{l+k}\|^2_{\mathcal{F}_q(H)^{\otimes 2}}.
\end{align*}
Therefore, we have
\begin{align*}
 &\left\|\sum_{k=m+1}^{n-l}q^{(k-m)(l+k-m)}U_{\mathrm{(2,3)}}\Xi_{k,l+k,m}\right\|_{\mathcal{F}_q(H)^{\otimes 2}}\\
 &\le C_{|q|}\max_{1\le k \le n}\|\xi_k\|^2_{\mathcal{F}_q(H)^{\otimes 2}} 
\left(\sum_{k=m+1}^{n-l}|q|^{2(k-m)(l+k-m)}\right)^{\frac{1}{2}}.
\end{align*}
Combining all of the above estimates, we conclude
\begin{align*}
    &\left\|\sum_{k=1}^{n-1}\sum_{l=k+1}^n\sum_{m=0}^{k-1}A_{k,l,m}\right\|\\
    &\le C_{|q|}^2 \max_{1\le k \le n}\|\xi_k\|^2_{\mathcal{F}_q(H)^{\otimes 2}}  \sum_{m=0}^{n-2}\sum_{l=1}^{n-m-1}\left(\sum_{k=m+1}^{n-l}|q|^{2(k-m)(l+k-m)}\right)^{\frac{1}{2}}\\
    &\le  C_{|q|}^2 \max_{1\le k \le n}\|\xi_k\|^2_{\mathcal{F}_q(H)^{\otimes 2}}  \sum_{m=0}^{n-2}\sum_{l=1}^{n-m-1}|q|^l\left(\sum_{k=m+1}^{n-l}|q|^{2(k-m)(k-m)}\right)^{\frac{1}{2}}\\
    &\le D_1 n \max_{1\le k \le n}\|\xi_k\|^2_{\mathcal{F}_q(H)^{\otimes 2}},
\end{align*}
where $D_1=C_{|q|}^2  \left(\sum_{l=1}^{\infty}|q|^{l}\right)\sqrt{\sum_{k=1}^{\infty} q^{2k^2}}$.
\end{proof}
The next step is to estimate $\left\|\sum_{k_1=1}^{n-2}\sum_{k_2=k_1+1}^{n-1}B_{k_1}^*B_{k_2}\right\|$ and $\sum_{k=1}^{n-1}\sum_{m=0}^{n-k-1}\|B_{k,m}\|$.
\begin{lm}\label{step3}
There exists $D_2=D_2(|q|)>0$ such that
    \begin{align*}
     \sum_{k=1}^{n-1}\sum_{m=0}^{n-k-1}\|B_{k,m}\|&\le D_2 n \max_{1 \le k \le n} \left\| \xi_k \right\|_{\mathcal{F}_q(H)^{\otimes 2}}^4\\
       \left\|\sum_{k_1=1}^{n-2}\sum_{k_2=k_1+1}^{n-1}B_{k_1}^*B_{k_2}\right\|&\le D_2 n^2 \max_{1 \le k \le n} \left\| \xi_k \right\|_{\mathcal{F}_q(H)^{\otimes 2}}^4
    \end{align*}
\end{lm}
\begin{proof}
 Recall that $B_{k,m}$ is equal to
 \[\sum_{l_1,l_2=k+1}^{n}q^{(n-k-m)^2} M(a_{n-l_1}\otimes a^*_{l_1-k}\otimes a_{n-k-m} \otimes a^*_{n-k-m} \otimes a_{l_2-k}\otimes a^*_{n-l_2}) (\Xi_{m})\]
  
  where $\Xi_m$ is 
  \[U_{(3,4)}(I \otimes \Phi_m \otimes I)(I\otimes R^*_{n-k-m,n-k}\otimes R^*_{m,n-k} \otimes I)(\Xi^*_{k,l_1}\otimes \Xi_{k,l_2}) \]
  which is a vector in $ H_2^{\otimes n-l_1}\otimes H_1^{\otimes l_1-k} \otimes H_2^{\otimes n-k-m} \otimes H_2^{\otimes n-k-m}\otimes H_1^{\otimes l_2-k} \otimes H_2^{\otimes n-l_2}$. Since $H_1$ is orthogonal to $H_2$, the $q$-commutation relations tell us, for any non-negative integers $k,l$,
  \begin{align*}
   M(a^*_k\otimes a_l)|_{H_1^{\otimes k}\otimes H_2^{\otimes l}}&=q^{kl}M(a_l\otimes a^*_k)U_{\mathrm{flip}}|_{H_1^{\otimes k}\otimes H_2^{\otimes l}}\\
   M(a^*_k\otimes a_l)|_{H_2^{\otimes k}\otimes H_1^{\otimes l}}&=q^{kl}M(a_l\otimes a^*_k)U_{\mathrm{flip}}|_{H_2^{\otimes k}\otimes H_1^{\otimes l}}.
   \end{align*}
   In particular, we have
   \begin{align*}
       &M(a_{n-l_1}\otimes a^*_{l_1-k}\otimes a_{n-k-m} \otimes a^*_{n-k-m} \otimes a_{l_2-k}\otimes a^*_{n-l_2}) (\Xi_{m})\\
       &= q^{\sum_{i=1}^2(l_i-k)(n-k-m)}M(a_{n-l_1}\otimes  a_{n-k-m}\otimes a^*_{l_1-k} \otimes   a_{l_2-k}\otimes a^*_{n-k-m}\otimes a^*_{n-l_2}) (\Xi'_{m})
   \end{align*}
   where $\Xi'_{m}=U_{(2,3)(4,5)}\Xi_m$ and $U_{(2,3)(4,5)}$ acts as the permutaion $(2,3)(4,5)$ on $ H_2^{\otimes n-l_1}\otimes H_1^{\otimes l_1-k} \otimes H_2^{\otimes n-k-m} \otimes H_2^{\otimes n-k-m}\otimes H_1^{\otimes l_2-k} \otimes H_2^{\otimes n-l_2}$.
   By applying Corollary \ref{creationannihilation2} to the product of $a^*_{l_1-k}$ and $a_{l_2-k}$, and we have
   \begin{align*}
       &M(a_{n-l_1}\otimes  a_{n-k-m}\otimes a^*_{l_1-k} \otimes   a_{l_2-k}\otimes a^*_{n-k-m}\otimes a^*_{n-l_2}) (\Xi'_{m})\\
       &=\sum_{m'=0}^{\min(l_1-k,l_2-k)}q^{\prod_{i=1}^2(l_i-k-m')}\\
       &\qquad M(a_{n-l_1}\otimes  a_{n-k-m}\otimes  a_{l_2-k-m'}\otimes a^*_{l_1-k-m'}  \otimes a^*_{n-k-m}\otimes a^*_{n-l_2}) (\Xi''_{m,m'})
   \end{align*}
   where $\Xi''_{m,m'}$ is 
   \[ U_{(3,4)} (I\otimes \Phi_{m'}\otimes I)(I\otimes R^*_{l_1-k-m',l_1-k}\otimes R^*_{m',l_2-k} \otimes I)(\Xi'_{m}),\]
   which is in $ H_2^{\otimes n-l_1}\otimes H_2^{\otimes n-k-m}  \otimes  H_1^{\otimes l_2-k-m'}\otimes  H_1^{\otimes l_1-k-m'} \otimes H_2^{\otimes n-k-m}\otimes H_2^{\otimes n-l_2}$, and $U_{(3,4)}$ is the flip of $ H_1^{\otimes l_1-k-m'}$ and $H_1^{\otimes l_2-k-m'}$ in this Hilbert space. Note that $(I\otimes \Phi_{m'} \otimes I)$ acts on 
   \[H_2^{\otimes n-l_1}\otimes H_2^{\otimes n-k-m}  \otimes  H_1^{\otimes l_1-k-m'}\otimes H_1^{\otimes m'} \otimes H_1^{\otimes m'}\otimes  H_1^{\otimes l_2-k-m'} \otimes H_2^{\otimes n-k-m}\otimes H_2^{\otimes n-l_2} \]
   where we apply $\Phi_{m'}$ to the fourth and fifth tensor components and the identity operators to other tensor components. 
   By Corollary \ref{generalproduct}, we have
   \begin{align*}
      & \left\|M(a_{n-l_1}\otimes  a_{n-k-m}\otimes  a_{l_2-k-m'}\otimes a^*_{l_1-k-m'}  \otimes a^*_{n-k-m}\otimes a^*_{n-l_2}) (\Xi''_{m,m'})\right\|\\
      &\le C_{|q|} \|\Xi''_{m,m'} \|_{\mathcal{F}_q(H)^{\otimes 2}}
   \end{align*}
   where we consider the norm on the right-hand side in the tensor product of two Hilbert spaces $H_2^{\otimes n-l_1}\otimes H_2^{\otimes n-k-m}  \otimes  H_1^{\otimes l_2-k-m'}\subset H^{\otimes 2n-2k-l_1+l_2-m-m'}$ and $H_1^{\otimes l_1-k-m'} \otimes H_2^{\otimes n-k-m}\otimes H_2^{\otimes n-l_2}\subset H^{\otimes 2n-2k+l_1-l_2-m-m'}$. By Lemma \ref{NCbinom}, we have
   \begin{align*}
       P^{(2n-2k-l_1+l_2-m-m')}&\le C_{|q|} P^{(n-l_1)}\otimes P^{(n-2k+l_2-m-m')}\\
       P^{(2n-2k+l_1-l_2-m-m')}&\le C_{|q|}P^{(n-2k+l_1-m-m')}\otimes P^{(n-l_2)}.
   \end{align*}
   Thanks to these inequalities and the orthogonality of $H_1$ and $H_2$, we have
   \[  \|\Xi''_{m,m'} \|_{\mathcal{F}_q(H)^{\otimes 2}}\le C_{|q|}\|\Xi''_{m,m'} \|_{\mathcal{F}_q(H)^{\otimes 6}}.\]
   Since $U_{(3,4)}$ is unitary on $\mathcal{F}_q(H)^{\otimes 6}$, we have
   \[\|\Xi''_{m,m'} \|_{\mathcal{F}_q(H)^{\otimes 6}}=\| (I\otimes \Phi_{m'}\otimes I)(I\otimes R^*_{l_1-k-m',l_1-k}\otimes R^*_{m',l_2-k} \otimes I)(\Xi'_{m})\|_{\mathcal{F}_q(H)^{\otimes 6}}.\]
   We use the identity
   \begin{align*}
   & (I\otimes R^*_{l_1-k-m',l_1-k}\otimes R^*_{m',l_2-k} \otimes I)U_{(2,3)(4,5)(3,4)}(I \otimes \Phi_m \otimes I)\\
   &=U_{(2,3)(4,5)(3,4)}(I \otimes \Phi_m \otimes I)(I\otimes R^*_{l_1-k-m',l_1-k}\otimes I \otimes R^*_{m',l_2-k} \otimes I)
\end{align*}
where $U_{(2,3)(4,5)(3,4)}=U_{(2,3)(4,5)}U_{(3,4)}$ and $(I\otimes R^*_{l_1-k-m',l_1-k}\otimes I \otimes R^*_{m',l_2-k} \otimes I)$ acts on 
\[ H_2^{\otimes n-l_1}\otimes    H_1^{\otimes l_1-k}\otimes H_2^{\otimes n-k-m}\otimes H_{2}^{\otimes m}\otimes H_{2}^{\otimes m} \otimes H_2^{\otimes n-k-m}\otimes H_1^{\otimes l_2-k} \otimes H_2^{\otimes n-l_2} \]
where $R^*_{l_1-k-m',l_1-k}$ acts on the second tensor component and $  R^*_{m',l_2-k}$ acts on the seventh tensor component and the identity operator $I$ acts on other tensor components.
By using this identity, $(I\otimes \Phi_{m'}\otimes I)(I\otimes R^*_{l_1-k-m',l_1-k}\otimes R^*_{m',l_2-k} \otimes I)(\Xi'_{m})$ can be written in the following form
\[(I\otimes \Phi_{m'}\otimes I)U_{(2,3)(4,5)(3,4)}(I\otimes \Phi_m \otimes I)(\xi \otimes \eta) \]
with
\begin{align*}
    \xi&=(I\otimes R^*_{l_1-k-m',l_1-k}\otimes R^*_{n-k-m,n-k})(\Xi^*_{k,l_1}) \\
    \eta&=(R^*_{m,n-k} \otimes R^*_{m',l_2-k} \otimes I)(\Xi_{k,l_2}).
\end{align*}
Note that 
\begin{align*}
    \xi &\in H_2^{\otimes n-l_1} \otimes H_1^{\otimes l_1-k-m'} \otimes H_1^{\otimes m'} \otimes H_2^{\otimes n-k-m} \otimes H_2^{\otimes m}\\
    \eta &\in H_2^{\otimes m} \otimes H_2^{\otimes n-k-m} \otimes H_1^{\otimes m'} \otimes H_1^{\otimes l_2-k-m'} \otimes H_2^{\otimes n-l_2}.
\end{align*}
For this type of vector, we have the following estimates 
   \begin{lm}\label{doublePhi}
   For $\xi \in H^{\otimes n-l_1} \otimes H^{\otimes l_1-k-m'} \otimes H^{\otimes m'} \otimes H^{\otimes n-k-m} \otimes H^{\otimes m}$ and $\eta \in H^{\otimes m} \otimes H^{\otimes n-k-m} \otimes H^{\otimes m'} \otimes H^{\otimes l_2-k-m'} \otimes H^{\otimes n-l_2}$, we have
   \begin{align*}
      & \|(I\otimes \Phi_{m'}\otimes I)U_{(2,3)(4,5)(3,4)}(I\otimes \Phi_m \otimes I)(\xi \otimes \eta)\|_{\mathcal{F}_q(H)^{\otimes 6}}\\
       &\le \|\xi\|_{\mathcal{F}_q(H)^{\otimes 5}} \|\eta\|_{\mathcal{F}_q(H)^{\otimes 5}}
       \end{align*}
       \end{lm}
We delay the proof of Lemma \ref{doublePhi} to the end of this proof, on page \pageref{lemma3.5pf}.

       Thus, using Lemma \ref{doublePhi} we have
       \[\|\Xi''_{m,m'} \|_{\mathcal{F}_q(H)^{\otimes 6}} \le \| \xi \|_{\mathcal{F}_q(H)^{\otimes 5}}\|\eta\|_{\mathcal{F}_q(H)^{\otimes 5}}. \]
       By using Lemma \ref{Rstar} and the fact $\ast$ is anti-unitary, we have
       \begin{align*}
         \| \xi \|_{\mathcal{F}_q(H)^{\otimes 5}}&\le C_{|q|}^{\frac{1}{2}} \|(I\otimes R^*_{n-k-m,n-k})(\Xi^*_{k,l_1})\|_{\mathcal{F}_q(H)^{\otimes 4}} \le C_{|q|}\|\Xi_{k,l_1}\|_{\mathcal{F}_q(H)^{\otimes 3}} \\
       \|\eta\|_{\mathcal{F}_q(H)^{\otimes 5}}&\le C_{|q|}^{\frac{1}{2}}\|(R^*_{m,n-k} \otimes I)(\Xi_{k,l_2})\|_{\mathcal{F}_q(H)^{\otimes 4}}\le C_{|q|}\|\Xi_{k,l_2}\|_{\mathcal{F}_q(H)^{\otimes 3}}
       \end{align*}
      By using Lemma \ref{tensorineq} and \ref{Rstar}, we have for $i=1,2$
       \begin{align}\label{Xi}
\|\Xi_{k,l_i}\|_{\mathcal{F}_q(H)^{\otimes 3}}&= \|(I_{n-k}\otimes \Phi_k \otimes I_{n-k})(I_{n}\otimes R^*_{k,l_i}\otimes I_{n-l_i})(\xi_k^*\otimes \xi_{l_i}) \|_{\mathcal{F}_q(H)^{\otimes 3}} \\
&\le \|\xi_k\|_{\mathcal{F}_q(H)^{\otimes 2}}\|(R^*_{k,l_i}\otimes I_{n-l_i})(\xi_{l_i}) \|_{\mathcal{F}_q(H)^{\otimes 3}} \\
&\le C_{|q|}^{\frac{1}{2}}\|\xi_k\|_{\mathcal{F}_q(H)^{\otimes 2}}
\|\xi_{l_i}\|_{\mathcal{F}_q(H)^{\otimes 2}}.
\end{align}
By combining all these estimates, we obtain
\begin{align*}
     &\sum_{k=1}^{n-1}\sum_{m=0}^{n-k-1}\|B_{k,m}\|\\
     &\le C_{|q|}^{5}\max_{1\le k\le n} \|\xi_k\|^4_{\mathcal{F}_q(H)^{\otimes 2}} \\
     &\qquad \cdot \sum_{k=1}^{n-1}\sum_{m=0}^{n-k-1}\sum_{l_1,l_2=k+1}^n\sum_{m'=0}^{\min(l_1-k,l_2-k)} |q|^{(n-k-m)^2+\sum_{i=1}^2(l_i-k)(n-k-m)+\prod_{j=1}^2(l_j-k-m')}\\
     &\le D_2 n \max_{1\le k\le n} \|\xi_k\|^4_{\mathcal{F}_q(H)^{\otimes 2}}
\end{align*}
where $D_2 = C_{|q|}^5 \left(\sum_{k=0}^{\infty} |q|^{k^2} \right)^2\left(\sum_{k=0}^{\infty} |q|^{k} \right)^2$.
For the second inequality in this Lemma, we can use the same argument (the difference is that we take $k_1,k_2$ instead of a single $k$) to get the factor $\max_{1\le k\le n} \|\xi_k\|^4_{\mathcal{F}_q(H)^{\otimes 2}} $, and the remaining part can be written as follows. 
\begin{align*}
    & C_{|q|}^{5} \sum_{k_1=1}^{n-2}\sum_{k_2=k_1+1}^{n-1}\sum_{\substack{l_1=k_1+1\\ l_2=k_2+1}}^n \sum_{m=0}^{n-k_2}\sum_{m'=0}^{\min(l_1-k_1,l_2-k_2)} \\
      &\quad |q|^{(l_1-k_1)(n-k_2-m)+(l_2-k_2)(n-k_1-m)+\prod_{i=1}^2(n-k_i-m)+\prod_{i=1}^2(l_i-k_i-m')},
\end{align*}
which is bounded by $D_2 n^2$.  This concludes the proof of Lemma \ref{step3}.
\end{proof}

\begin{proof}[Proof of Lemma \ref{doublePhi}]\label{lemma3.5pf}

    For $\xi=\sum x_1\otimes x_2 \otimes x_3 \otimes x_4 \otimes x_5$ and $\eta=\sum y_5\otimes y_4 \otimes y_3 \otimes y_2 \otimes y_1$, we have
    \begin{align*}
      &(I\otimes \Phi_{m'}\otimes I)U_{(2,3)(4,5)(3,4)}(I\otimes \Phi_m \otimes I)(\xi \otimes \eta)\\
      &=\sum \Phi_m(x_5\otimes y_5)\Phi_{m'}(x_3\otimes y_3)x_1 \otimes y_4 \otimes x_2 \otimes y_2 \otimes x_4 \otimes y_1  \\
      &=\sum \langle x_5, y_5^*\rangle_q\langle x_3, y_3^*\rangle_qx_1 \otimes y_4 \otimes x_2 \otimes y_2 \otimes x_4 \otimes y_1\\
      &=\sum\langle x_3\otimes x_5, (y_5\otimes y_3)^*\rangle_{\mathcal{F}_q(H)^{\otimes 2}} x_1 \otimes y_4 \otimes x_2 \otimes y_2 \otimes x_4 \otimes y_1. 
    \end{align*}
    By taking the norm of this sum and permuting the tensor components, we get
 \begin{align*}
 &\left\|\sum\langle x_3\otimes x_5, (y_5\otimes y_3)^*\rangle_{\mathcal{F}_q(H)^{\otimes 2}} x_1 \otimes y_4 \otimes x_2 \otimes y_2 \otimes x_4 \otimes y_1 \right\|_{\mathcal{F}_q(H)^{\otimes 6}}\\
 &= \left\|\sum\langle x_3\otimes x_5, (y_5\otimes y_3)^*\rangle_{\mathcal{F}_q(H)^{\otimes 2}} x_1 \otimes   x_2 \otimes x_4 \otimes y_4 \otimes  y_2 \otimes y_1\right\|_{\mathcal{F}_q(H)^{\otimes 6}}\\
 &= \left\| (I\otimes \Phi \otimes I)\left(\sum x_1\otimes x_2 \otimes x_4 \otimes x_3 \otimes x_5\right)\otimes \left(\sum y_5 \otimes y_3 \otimes y_4 \otimes y_2 \otimes y_1\right)\right\|_{\mathcal{F}_q(H)^{\otimes 6}}
 \end{align*}
 where $\Phi: (H^{m'}\otimes H^{m})\otimes (H^{m'}\otimes H^{m}) \to \mathbb{C}$ is defined by $\Phi[(\xi_1 \otimes \xi_2)\otimes (\eta_2 \otimes \eta_1)] =\langle \xi_1 \otimes \xi_2, (\eta_2 \otimes \eta_1)^*\rangle_{\mathcal{F}_q(H)^{\otimes 2}}$. Since $\ast$ is anti-unitary, we apply Lemma \ref{tensorineq} and we get
 \begin{align*}
   &\left\| (I\otimes \Phi \otimes I)\left(\sum x_1\otimes x_2 \otimes x_4 \otimes x_3 \otimes x_5\right)\otimes \left(\sum y_5 \otimes y_3 \otimes y_4 \otimes y_2 \otimes y_1\right)\right\|_{\mathcal{F}_q(H)^{\otimes 6}}\\
   &\le \left\| \sum x_1\otimes x_2 \otimes x_4 \otimes x_3 \otimes x_5 \right\|_{\mathcal{F}_q(H)^{\otimes 5}}
   \left\| \sum  y_5 \otimes y_3 \otimes y_4 \otimes y_2 \otimes y_1\right\|_{\mathcal{F}_q(H)^{\otimes 5}}\\
   &=\|\xi\|_{\mathcal{F}_q(H)^{\otimes 5}} \|\eta \|_{\mathcal{F}_q(H)^{\otimes 5}}
 \end{align*}
\end{proof}
To estimate the last term $\sum_{k=1}^{n-1}\|B_{k,n-k}\|$, we use a strong induction argument. 
\begin{lm}\label{step4}
    If we assume Lemma \ref{KeyLem} holds for all $1\le m < n$, then we have
    \[\sum_{k=1}^{n-1}\|B_{k,n-k}\| \le \frac{A^2C_{|q|}}{2} n^2 \max_{1\le k \le n}\left\|\xi_k \right\|_{\mathcal{F}_q(H)^{\otimes 2}}^4.\]
\end{lm} 
\begin{proof}
 Recall that $B_{k,n-k}$ is equal to 
   \begin{align*}
    \sum_{l_1,l_2=k+1}^n M(a_{n-l_1}\otimes a^*_{l_1-k} \otimes a_{l_2-k}\otimes a^*_{n-l_2})(I\otimes \Phi_{n-k}\otimes I)(\Xi_{k,l_1}^*\otimes \Xi_{k,l_2})
   \end{align*}
   where $I\otimes \Phi_{n-k}\otimes I$ acts on $H_2^{\otimes n-l_1}\otimes H_1^{\otimes l_1-k}\otimes H_2^{\otimes n-k}\otimes H_2^{\otimes n-k}\otimes H_1^{\otimes l_2-k}\otimes H_2^{\otimes n-l_2}$ where $\Phi_{n-k}$ acts on the third and fourth tensor components and the identity operator $I$ acts on the other tensor components. 
   By using the orthonormal basis $\{f_s\}_{s \in S}$ of $H_2^{\otimes n-k}$ with respect to the $q$-inner product, we decompose $\Phi_{n-k}$ as follows:
   \[ \Phi_{n-k}(\xi \otimes \eta)= \sum_{s \in S} \langle \xi , f_s\rangle_q \langle f_s , \eta^* \rangle_q=\psi_s(\xi)\phi_s(\eta) \]
where $\psi_s(\xi)= \langle \xi,f_s\rangle_q$ and $\phi_s(x)=\langle f_s , \eta^* \rangle_q$. Note that $\psi_s$ and $\phi_s$ are linear functionals on $H^{\otimes n-k}$ and
\[\overline{\phi_s(\xi^*)}=\overline{\langle f_s , \xi \rangle_q}=\langle \xi , f_s\rangle_q=\psi_s(\xi).\]
Thus we can also see,
\begin{align*}
(I \otimes\psi_s)(\xi\otimes \eta)^*&=\psi_s(\xi^*) \eta^*\\
&=(\overline{\psi_s(\xi^*)}\eta)^*\\
&=(\phi_s(\xi)\eta)^*\\
&=[(\phi_s\otimes I)(\xi\otimes \eta)]^*
\end{align*}
   By using this decomposition,
   $B_{k,n-k}$ is equal to
   \begin{align*}
      & \sum_{s\in S}\sum_{l_1,l_2=k+1}^n M(a_{n-l_1}\otimes a^*_{l_1-k} \otimes a_{l_2-k}\otimes a^*_{n-l_2})\left[(I\otimes \psi_s)(\Xi_{k,l_1}^*)\right]\otimes\left[ (\phi_s\otimes I) (\Xi_{k,l_2})\right]
      \\
       &=\sum_{s\in S}\sum_{l_1,l_2=k+1}^n M(a_{n-l_1}\otimes a^*_{l_1-k} \otimes a_{l_2-k}\otimes a^*_{n-l_2})\left[(\phi_s \otimes I) \Xi_{k,l_1}\right]^* \otimes (\phi_s \otimes I)\Xi_{k,l_2}\\
       &=\sum_{s\in S}\left[ \sum_{l=k+1}^n M(a_{l-k}\otimes a^*_{n-l})(\phi_s  \otimes I)\Xi_{k,l}\right]^*\sum_{l=k+1}^n M(a_{l-k}\otimes a^*_{n-l})(\phi_s \otimes I)\Xi_{k,l}.
       \end{align*}

Therefore, we have
       
       \[
           \sum_{k=1}^{n-1}\|B_{k,n-k}\| \le \sum_{k=1}^{n-1}\sum_{s \in S}\left\| \sum_{l=k+1}^n M(a_{l-k}\otimes a^*_{n-l})(\phi_s\otimes I)\Xi_{k,l}\right\|^2.
       \]
       where $(\phi_s \otimes I)\Xi_{k,l}\in H_1^{\otimes l-k}\otimes H_2^{\otimes n-l}  $. Thus we can use the assumption of induction, and we get
       \begin{align*}
           \left\| \sum_{l=k+1}^n M(a_{l-k}\otimes a^*_{n-l})(\phi_s \otimes I)\Xi_{k,l}\right\|^2\le A^2(n-k)\max_{1\le l \le n-k}\left\|(\phi_s \otimes I)\Xi_{k,l+k}\right\|^2_{\mathcal{F}_q(H)^{\otimes 2}}
       \end{align*}
       Now, we use the following lemma.
\begin{lm}\label{estimatephi}
    Under the setting above, we have
    \[
           \sum_{s \in S}\left\|(\phi_s \otimes I)\Xi_{k,l+k}\right\|^2_{\mathcal{F}_q(H)^{\otimes 2}} = \left\|\Xi_{k,l+k}\right\|^2_{\mathcal{F}_q(H)^{\otimes 3}}
       \]
\end{lm}
       
By using this lemma and the inequalities in (\ref{Xi}) in the proof of Lemma \ref{step3}, we have
    \[
      \sum_{s \in S}\left\|(\phi_s \otimes I)\Xi_{k,l+k}\right\|^2_{\mathcal{F}_q(H)^{\otimes 2}} = \left\|\Xi_{k,l+k}\right\|^2_{\mathcal{F}_q(H)^{\otimes 3}}\le C_{|q|}\|\xi_{k}\|^2_{\mathcal{F}_q(H)^{\otimes 2}}  \|\xi_{l+k}\|^2_{\mathcal{F}_q(H)^{\otimes 2}}  
    \]
           Thus, we obtain
           \begin{align*}
           \sum_{k=1}^{n-1}\|B_{k,n-k}\| &\le A^2 C_{|q|}\max_{1\le k \le n}\|\xi_k\|^4_{\mathcal{F}_q(H)^{\otimes 2}} \sum_{k=1}^{n-1}(n-k)\\
           &\le \frac{A^2 C_{|q|}}{2} n^2 \max_{1\le k \le n}\|\xi_k\|^4_{\mathcal{F}_q(H)^{\otimes 2}}
           \end{align*}
\end{proof}
\begin{proof}[Proof of Lemma \ref{estimatephi}]
    Recall that $\Xi_{k,l+k} \in H_2^{\otimes n-k} \otimes H_1^{\otimes l}\otimes H_2^{\otimes n-l-k}$ and $\{f_s\}_{s \in S}$ is an orthonormal basis of $H_2^{\otimes n-k}$. When we write $\Xi_{k,l+k}= \sum_i x_i \otimes y_i \otimes z_i$, we have
    
    \begin{align*}
        \sum_{s \in S}\left\|(\phi_s \otimes I)\Xi_{k,l+k} \right\|^2_{\mathcal{F}_q(H)^{\otimes 2}} &= \sum_{s \in S}\sum_{i,i'}\langle f_s, x_i^* \rangle_q \langle x_{i'}^*, f_s \rangle_q \langle y_i,y_{i'}\rangle_q\langle z_i,z_{i'}\rangle_q \\
        &=\sum_{i,i'}\langle x_{i'}^*, x_i^* \rangle_q \langle y_i,y_{i'}\rangle_q \langle z_i,z_{i'}\rangle_q\\
        &=\sum_{i,i'}\langle x_{i}, x_{i'} \rangle_q \langle y_i,y_{i'}\rangle_q\langle z_i,z_{i'}\rangle_q=\left\| \Xi_{k,l+k}\right\|^2_{\mathcal{F}_q(H)^{\otimes 3}},
    \end{align*}
    where we use the fact that $\ast$ is anti-unitary.
\end{proof}
\begin{proof}[Proof of Lemma \ref{KeyLem}]
    We prove this Lemma by induction. When $n=1$, we have the inequality by Corollary \ref{norm_of_prodcuts}. If Lemma \ref{KeyLem} holds for $1 \le k \le n-1$, then by Lemma \ref{step1}, \ref{step2}, we have
    \begin{align*}
        & \left\| \sum_{k=1}^nM (a_k\otimes a^*_{n-k})\xi_k\right\|^2\\
    &\le  C_{|q|}^2 n \max_{1\le k \le n} \left\| \xi_k \right\|_{\mathcal{F}_q(H)^{\otimes 2}}^2  + 2\left\|\sum_{k=1}^{n-1}\sum_{l=k+1}^n A_k^* A_l \right\|\\
    &\le (C_{|q|}^2+2D_1) n \max_{1\le k \le n} \left\| \xi_k \right\|_{\mathcal{F}_q(H)^{\otimes 2}}^2 + 2\left\|\sum_{k=1}^{n-1}\sum_{l=k+1}^nA_{k,l,k}\right\|. 
    \end{align*}
    By using Lemma \ref{step3}, \ref{step4}, $\left\|\sum_{k=1}^{n-1}\sum_{l=k+1}^nA_{k,l,k}\right\|$ is bounded by
    \begin{align*}
    & \sqrt{(2D_2 n^2+D_2 n) \max_{1 \le k \le n}\left\| \xi_k \right\|_{\mathcal{F}_q(H)^{\otimes 2}}^4 + \sum_{k=1}^{n-1}\|B_{k,n-k}\|}\\
    &\le \sqrt{3D_2+\frac{A^2C_{|q|}}{2}}
   \ n \max_{1 \le k \le n} \left\| \xi_k\right\|_{\mathcal{F}_q(H)^{\otimes 2}}^2.
    \end{align*}
    By combining them, $\left\| \sum_{k=1}^nM (a_k\otimes a^*_{n-k})\xi_k\right\|^2$ is bounded by 
    \[\left(C_{|q|}^2+2D_1+2\sqrt{3D_2+\frac{A^2C_{|q|}}{2}} \right) n \max_{1 \le k \le n} \left\| \xi_k\right\|_{\mathcal{F}_q(H)^{\otimes 2}}^2 \]
    Therefore, we have the desired inequality if we take $A$ large enough to have 
    \[ C_{|q|}^2+2D_1+2\sqrt{3D_2+\frac{A^2C_{|q|}}{2}} \le A^2. \]
\end{proof}

\begin{proof}[Proof of Theorem \ref{Mainresult}]
   By Lemma \ref{qCircular}, we have
   \[ \sum_{|w|=n} \alpha_w c_w^{(q)}
  =\sum_{k=0}^n M(a_k\otimes a^*_{n-k}) (I_k\otimes \overline{I}_{n-k})R^*_{k,n} \xi\]
  where $\xi = \sum_{|w|=n} \alpha_w e_w \in H_1^{\otimes n}$.
  By using the triangle inequality, 
  \begin{align*}&\left\|\sum_{k=0}^n M(a_k\otimes a^*_{n-k}) (I_k\otimes \overline{I}_{n-k})R^*_{k,n} \xi\right\|\\
  &\le \|a^*_n(\overline{I}_n (\xi))\|+\left\|\sum_{k=1}^n M(a_k\otimes a^*_{n-k}) (I_k\otimes \overline{I}_{n-k})R^*_{k,n} \xi\right\|. 
  \end{align*}
  Applying Lemma \ref{KeyLem} to the second term, we get
  \[\left\|\sum_{k=1}^n M(a_k\otimes a^*_{n-k}) (I_k\otimes \overline{I}_{n-k})R^*_{k,n} \xi\right\|\le A\sqrt{n} \max_{1\le k \le n}\|(I_k\otimes \overline{I}_{n-k})R^*_{k,n} \xi\|_{\mathcal{F}_q(H)^{\otimes 2}}. \]
  Since $I_k\otimes \overline{I}_{n-k}$ is unitary, by Lemma \ref{Rstar}, we have
  \[\|(I_k\otimes \overline{I}_{n-k})R^*_{k,n} \xi\|_{\mathcal{F}_q(H)^{\otimes 2}}\le C_{|q|}^{\frac{1}{2}}\|\xi\|_{\mathcal{F}_q(H)}, \]
  and by Lemma \ref{norm_of_creation},
  \[ \|a^*_n(\overline{I}_n (\xi))\|\le C_{|q|}^{\frac{1}{2}}\|\xi\|_{\mathcal{F}_q(H)}.\]
  Thus, we conclude
  \begin{align*}
   \left\|\sum_{|w|=n} \alpha_w c_w^{(q)}\right\|&\le C_{|q|}^{\frac{1}{2}}(A\sqrt{n}+1) \|\xi\|_{\mathcal{F}_q(H)} \\
   &\le A'\sqrt{n+1} \left\|\sum_{|w|=n} \alpha_w c_w^{(q)}\right\|_2
  \end{align*}
   for some $A'=A'(|q|)$ (for example $A'=A \sqrt{2C_{|q|}}$).
 \end{proof}

\section{An Application: Strong Ultracontractivity\label{sect.ultra}}

As an application of the strong Haagerup inequality that we obtained, we now prove the strong ultracontractivity of the $q$-Ornstein-Uhlenbeck semigroup, based on the argument in \cite[Section 5]{MR2353703}.  The reader may wish to review the $q$-Fock space construction at the beginning of Section \ref{preliminary}.

Given a Hilbert space $H$ and a vector $\xi\in H$, the associated ``field operator'' $X^{(q)}(\xi)$ is defined as
\[ X^{(q)}(\xi) := a(\xi)+a^\ast(\xi) \quad \text{acting on} \quad \mathcal{F}_q(H). \]
These operators are all selfadjoint and bounded for $-1\le q<1$.  If $\{e_i\}$ is an orthonormal basis for $H$, then the operators $X^{(q)}_i = X^{(q)}(e_i)$ generate a von Neumann algebra known as the {\bf $q$-Gaussian algebra} $\Gamma_q(H)$.  (Any two orthonormal bases for $H$ generate the same algebra, and the isometry of $H$ between the two bases induces a $\ast$-automorphism of $\Gamma_q(H)$.)  The vacuum expectation state $\tau(X) = \langle X\Omega,\Omega \rangle_{\mathcal{F}_q}$ is a faithful, normal trace on $\Gamma_q(H)$, which is a $\mathrm{II}_1$-factor.  In terms of the state $\tau$, the law of each selfadjoint operator $X^{(q)}(\xi)$ with $\|\xi\|=1$ is the $q$-Gaussian (aka $q$-semicircular) distribution.

The faithfulness of $\tau$ on $\Gamma_q(H)$ shows that the map
\[ \Gamma_q(H)\to\mathcal{F}_q(H), \qquad X\mapsto X\Omega \]
is injective, and is (by definition of the inner product) an isometry with respect to the inner product $\langle X,Y\rangle_\tau = \tau(Y^\ast X)$.  Therefore, this map extends to an isometric isomorphism from $L^2(\Gamma_q(H),\tau)$ onto a closed subspace of $\mathcal{F}_q(H)$.  In fact, $\Gamma_q(H)\Omega$ is dense in $\mathcal{F}_q(H)$ (cf.\ \cite[Section 2]{MR1463036}) and so  $L^2(\Gamma_q(H),\tau)\cong\mathcal{F}_q(H)$.

The {\em number operator} $N$ is defined on the algebraic Fock space $\mathcal{F}_{\mathrm{alg}}(H)$ by $N\Omega=0$ and $N(\xi_1\otimes\cdots\otimes \xi_n) = n\,\xi_1\otimes\cdots\otimes\xi_n$ for $n\in\mathbb{N}$; it extends to a densely-defined self-adjoint operator on $\mathcal{F}_q(H)$ for $|q|\le 1$.  Intertwining with the $L^2$-isomorphism mentioned above, it induces a densely-defined selfadjoint operator $N_q$ on $L^2(\Gamma_q(H),\tau)$, whose spectrtum is $\mathbb{N}$ and is therefore a non-negative operator.  Thus $-N_q$ generates a contraction semigroup which (first appearing in \cite{MR1146011}) is known as the {\bf $q$-Ornstein--Uhlenbeck semigroup}, or {\bf $q$-OU semigroup} for short.

The $q=1$ version was thoroughly studied in the 1960s and 1970s, where it played an important role in constructive quantum field theory.  There, the space $L^2(\Gamma_1(H),\tau)$ is the space of $L^2$ functions with respect to a Gaussian cylinder measure $\gamma$ on $H$, and the number operator is the associated divergence form operator, i.e.\ satisfying $\langle N_1 \psi,\psi\rangle = -\int_H |\nabla \psi|^2\,d\gamma$.  If $H=\mathbb{C}^d$ then $N_1 = \Delta - x\cdot\nabla$ is the Ornstein--Uhlenbeck operator, ergo the name for general $q$.  The eigenfucntions of $N_1$ are tensor products of Hermite polynomials; there is a similar analysis of $N_q$'s eigenstates in terms of $q$-Hermite polynomials, see \cite[Section 2] {MR1463036}.

The OU semigroup $e^{-tN_1}$ has many smoothing properties that are useful tools in the analysis of physics-motivated PDE problems whose linearization involves the operator $N_1$.  These smoothing properties can be dually measured in terms of $L^p$ estimates for the semigroup.  As a Markov semigroup, it is of course a contraction on $L^p$ for $1\le p<\infty$.  More significantly, for $p\ge 1$, $e^{-tN_1}$ maps $L^p$ into $L^\infty$ for each $t>0$, and $e^{-tN_1}\colon L^p\to L^\infty$ is bounded.  This property is known as {\bf ultracontractivity}, cf.\ \cite{MR849557}.  It follows that $e^{-tN_1}$ maps $L^p$ into $L^r\supset L^\infty$ for any $r\ge p$, but that does not mean that it is bounded $L^p\to L^r$ (in general it is not).  Nevertheless, a further smoothing property the semigroup holds is {\bf hypercontractivity}: there is a finite time $t_N(p,r)$ such that $e^{-tN_1}\colon L^p\to L^r$ is bounded, in fact is a contraction, iff $t\ge t_N(p,r)$.  (The time to contraction is explicitly known to be $t_N(p,r) = \frac12\ln(\frac{r-1}{p-1})$, known as the {\em Nelson time}, cf.\ \cite{MR343816}.)

In \cite{MR1462754}, entitled {\em Free hypercontractivity}, Biane proved the same hypercontractivity estimates hold for the $q$-OU semigroup for $-1\le q<1$.  (His focus in the paper was the free probability $q=0$ case, but his proof and statement are general.)  That is: Biane proved that for any Hilbert space $H$, and any $1<p<r<\infty$, the $q$-OU semigroup $e^{-tN_q}\colon L^p(\Gamma_q(H),\tau)\to L^r(\Gamma_q(H),\tau)$ is a contraction iff $t\ge t_N(p,r)$ -- the same time to contraction independent of $q$.  Then, in \cite{MR1811255}, Bo\.zejko showed that the $q$-OU semigroup is ultracontractive: precisely, he proved that for any $h\in L^2(\Gamma_q(H),\tau)$ and any $t>0$, $e^{-tN_q}h$ is actually the in the von Neumann algebra $\Gamma_q(H)$ (``$=$'' $L^\infty$), and moreover
\begin{equation} \label{eq:Bo.uc} \|e^{-tN_q}\colon L^2(\Gamma_q(H),\tau)\to\Gamma_q(H)\| \le C\, t^{-3/2} \quad \text{for }t>0 \end{equation}
for some constant $C = C(q)<\infty$.  The self-adjointness of $e^{-tN_q}$ on $L^2$ then implies the same bound holds for the action $L^1\to L^2$, and then using the semigroup property yields
\[ \|e^{-tN_q}\|_{L^1\to L^\infty} \le \|e^{-\frac{t}{2}N_q}\|_{L^1\to L^2} \|e^{-\frac{t}{2}N_q}\|_{L^2\to L^\infty} \le \tilde{C} t^{-3} \quad \text{for }t>0 \]
where $\tilde{C} = 8C^2$.  It then follows by interpolation that the $q$-OU semigroup is bounded $L^p(\Gamma_q(H),\tau)\to \Gamma_q(H)$ for all $1\le p\le \infty$.  In particular, ultracontractivity is typically reduced to the question of boundedness $L^2\to L^\infty$, as in Bo\.zejko's estimate \eqref{eq:Bo.uc}.

Again in the classical $q=1$ setting, Janson \cite{MR706641} studied Nelson's hypercontractivity theorem with the domain restricted to {\em holomorphic} functions $L^p_{\mathrm{hol}}(\mathbb{C}^d,\gamma)$ in $L^p$ of the Gaussian cylinder measure.  The action of the OU-semigroup on this space is extremely simple: $e^{-tN_1}f(z) = f(e^{-t}z)$, i.e.\ holomorphic momnomials $z_1^{n_1}\ldots z_d^{n_d}$ are eigenfunctions for $N_1$ with eigenvalues $n_1+\cdots+n_d$.  Owing in-part to this dilation action, the hypercontractive estimates actually improve, in that the time to contraction shrinks: $\|e^{-tN_1}\colon L^p_{\mathrm{hol}}\to L^r\|\le 1$ iff $t\ge t_J(p,r) = \frac12\ln\frac{r}{p} < t_N(p,r)$.  This {\em strong hypercontractivity} theorem motivated the first author of the present paper to explore a holomorphic space version of Biane's $q$-Gaussian hypercontractivity estimates in \cite{MR2174419}.

Let $H = \mathbb{C}^d\oplus\mathbb{C}^d = H_1\oplus H_2$, and fix orthonormal bases $\{e_j\}_{j=1}^d$ for $H_1$ and $\{e_{\overline{j}}\}_{j=1}^d$ for $H_2$ as in Section \ref{preliminary}.  Let $c_j^{(q)} = a(e_j)+a(e_{\overline{j}})^\ast$ be $q$-circular elements as we have studied throughout this paper.  The von Neumann $\mathrm{W}^\ast(c_1^{(q)},\ldots,c_d^{(q)})$ is isomorphic to $\Gamma_q(H)$; indeed, as shown in \cite[Proposition 4]{MR2174419}, the unitary isomorphism
\[ H_1\oplus H_2\to H_1\oplus H_2, \quad (\xi,\eta)\mapsto \frac{1}{\sqrt{2}}(\xi+\eta,\xi-\eta) \]
induces a unitary isomorphism of $\mathcal{F}_q(H_1\oplus H_2)$ and thus an inner automorphism of $\mathscr{B}(\mathcal{F}_q(H))$; since it maps
\[ c^{(q)}_j\mapsto \frac{1}{\sqrt{2}}(X^{(q)}_j+iX^{(q)}_{\overline{j}}), \qquad 
(c^{(q)}_j)^\ast \mapsto \frac{1}{\sqrt{2}}(X^{(q)}_j-iX^{(q)}_{\overline{j}}) \]
it therefore restricts to a (trace preserving) isomorphism from $\mathrm{W}^\ast(c_1^{(q)},\ldots,c_d^{(q)})$ onto $\mathrm{W}^\ast(X^{(q)}_1,\ldots,X^{(q)}_{\overline{d}}) = \Gamma_q(H_1\oplus H_2) = \Gamma_q(H)$.  We can therefore transport all of the above results about the $q$-OU semigroup over $\Gamma_q(H)$ to the $q$-circular von Neumann algebra.

Let $\mathcal{H}_q(\mathbb{C}^d) = \mathbb{C}\langle c^{(q)}_1,\ldots,c^{(q)}_d\rangle$ denote the (non-selfadjoint) algebra of all noncommutative polynomial functions $\sum_{w}\alpha_w c_w^{(q)}$ in the $q$-circular generators (and {\em not} in their adjoints $c_1^{(q)*},\ldots,c_d^{(q)*}$).  In \cite{MR2174419}, the first author identified the intertwined closure $L^p_{\mathrm{hol}}(\mathcal{H}_q(\mathbb{C}^d),\tau)$ of $\mathcal{H}_q(\mathbb{C}^d)$ in $L^p( \Gamma_q(H),\tau)$ as a $q$-Gaussian version of the holomorphic Gaussian spaces $L^p_{\mathrm{hol}}(\mathbb{C}^d,\gamma)$ in \cite{MR706641}.  The intertwined number operator acts on $L^2(\mathcal{H}_q(\mathbb{C}^d),\tau)$ (identified with a subspace of $L^2(\Gamma_q(H),\tau)$) much as in the $q=1$ holomorphic setting:
 \[N_qc_w^{(q)}= |w| c_w^{(q)}, \quad w \in [d]^*.\]
Thus, on the algebra $\mathcal{H}_q(\mathbb{C}^d)$ of holomorphic polynomials in $c_1^{(q)},\ldots,c_d^{(q)}$, the $q$-OU semigroup $e^{-tN_q}$ coincides with the simple dilation semigroup $D_t$ introduced in \eqref{eq:dilation.semigroup}.

The main theorem \cite[Theorem 4]{MR2174419} proves that Janson's strong hypercontractivity estimates hold for the restriction of the $q$-OU semigroup to the case $L^2(\mathcal{H}_q(\mathbb{C}^d),\tau)\to L^r(\mathcal{H}_q(\mathbb{C}^d),\tau)$ for even integers $r$: here the time to contraction is the shorter Janson time $t_J(2,r) = \frac12\ln\frac{r}{2}$.  This strong hypercontractivity result for $q$-Gaussian holomorphic spaces led the first author and Speicher to explore whether other ``strong'' versions of norm inequalities hold when restricted to noncommutative holomorphic spaces like $\mathcal{H}_0$ (and $R$-diagonal generalizations), and this is what motivated the strong Haagerup inequality (in the $q=0$ case) proved in \cite{MR2353703}.  The final theorem  \cite[Theorem 5.4]{MR2353703} uses the strong Haagerup inequality to prove {\bf strong ultracontractivity} for the associated dilation semigroup, which agrees with the holomorphic restriction of the $0$-OU semigroup in the case of circular generators; see Theorem \ref{thm.ultra} for the precise definition of strong ultracontractivity.  Two following papers \cite{MR2564935,MR2661507} made these estimates sharper and also proved the lower bound.

Presently, we prove strong ultracontractivity for the $q$-OU semigroup for $|q|<1$.

\begin{thm} \label{thm.ultra}
For $-1<q<1$, the $q$-Ornstein-Uhleneck semigroup satisfies strong ultracontractivity, with a sharp divergence rate of $t^{-1}$ as $t\downarrow 0$.  Precise: there are constants $\alpha=\alpha(|q|),\beta=\beta(|q|)>0$ such that for $0<t<1$,
    \[\alpha t^{-1}\le\|e^{-tN_q}:L^2(\mathbb{C}\langle c^{(q)}_1,\ldots,c^{(q)}_d \rangle ,\tau)\to \mathrm{W}^*(c^{(q)}_1,\ldots,c^{(q)}_d )\|\le \beta t^{-1}. \]
The upper bound $\beta$ may be taken to equal $\frac12A$ where $A$ is the same constant in Theorem \ref{Mainresult}.
\end{thm}

\begin{proof}
The proof is based on the same idea in \cite[Theorem 3.18]{MR2564935}. Note that the subspaces:
\[ \mathcal{H}_q^{(n)} = \left\{\sum_{|w|=n}\alpha_w c^{(q)}_w\colon \alpha_w\in\mathbb{C}, w\in[d]^\ast\right\} \]
are orthogonal in $L^2(\mathcal{H}_q(\mathbb{C}^d),\tau)$.  Indeed, it follows quickly from the definition if the $q$-circular operators via the creation and annihilation operators in orthogonal spaces (at the beginning of Section \ref{preliminary}) that
\begin{equation} \label{eq: c.omega=e} c^{(q)}_w\Omega = e_w. \end{equation}
Since the map $X\mapsto X\Omega$ is an isometric isomorphism, $c^{(q)}_w$ and $c^{(q)}_{w'}$ are orthogonal whenever $\langle e_w,e_{w'}\rangle_{q} = 0$; in particular if $w$ and $w'$ have different lengths.  Since $\mathcal{H}_q(\mathbb{C}^d) = \mathbb{C}\langle c_1^{(q)},\ldots,c_d^{(q)}\rangle = \bigoplus_n \mathcal{H}_q^{(n)}$, it follows that every $h\in L^2(\mathcal{H}_q(\mathbb{C}^d),\tau)$ has a Hilbert space decomposition $h = \sum_n h_n$ with $h_n\in\mathcal{H}^{(n)}_q$.  Most importantly: our main Theorem \ref{Mainresult} asserts precisely that $\|h_n\|\le A\sqrt{n+1}\|h_n\|_2$.

Notice that, for any $h_n\in \mathcal{H}_q^{(n)}$, $N_q h_n =nh_n$,  and thus $e^{-tN_q}h_n = e^{-nt}h_n$.  Thus, if $h\in L^2(\mathcal{H}_q,\tau)$ then $h = \sum_n h_n$ and so $e^{-tN_q}h = \sum_n e^{-nt} h_n$.   We may then estimate the operator norm of $e^{-tN_q}h$ bluntly with the triangle inequality and Theorem \ref{Mainresult}:
\[ \|e^{-tN_q}h\| \le \sum_{n=0}^\infty e^{-nt}\|h_n\| \le \sum_{n=0}^\infty e^{-nt}\cdot A\sqrt{n+1}\|h_n\|_2. \]
Then, from the Cauchy--Schwarz inequality, we have
\begin{equation} \label{eq.C-S} \|e^{-tN_q}h\|^2 \le A^2 \sum_{n=0}^\infty (n+1)e^{-2nt} \cdot \sum_{n=0}^\infty \|h_n\|_2^2 \end{equation}
and the latter sum is simply $\|h\|_2^2$ by the orthogonality of the decomposition $h = \sum_n h_n$.  The first sum can be computed explicity with elementary calculus:
\[ \sum_{n=0}^\infty (n+1)e^{-2nt} = \sum_{n=0}^\infty e^{-2nt} -\frac12\sum_{n=0}^\infty \frac{d}{dt} e^{-2nt} = \frac{1}{(1-e^{-2t})^2}. \]
Hence, combining with \eqref{eq.C-S},
\[ \|e^{-tN_q}h\| \le \frac{A}{1-e^{-2t}}\|h\|_2 \qquad \text{for }\; h\in L^2(\mathcal{H}_q(\mathbb{C}^d),\tau). \]
The reader may readily verify that the function $t\mapsto t/(1-e^{-2t})$ is decreasing on $\mathbb{R}_+$ and has limit $\frac12$ as $t\downarrow 0$; thus, we have proved the upper bound in thee theorem.

We obtain a lower bound by applying $e^{-tN_q}$ to a $L^2$-vector generated by a single $q$-circular operator $c_1^{(q)}$. For simplicity, we use the notations $c,a,\overline{a},e,\overline{e}$ instead of $c_1^{(q)},a_1,a_{\overline{1}},e_1,e_{\overline{1}}$.  Since $\langle e^{\otimes n},e^{\otimes m}\rangle_q =\delta_{n,m}[n]_q!$ for $m,n \in \mathbb{Z}_{\ge 0}$ (cf.\ \cite[Theorem 2.1]{MR1811255}), the vector
\[\psi_t=\sum_{n=0}^{\infty}e^{-nt} \frac{e^{\otimes n}}{\sqrt{[n]_q!}} \]
is in $\mathcal{F}_q(\mathbb{C}\oplus\mathbb{C})$ for all $t>0$, with $\|\psi_t\|_{\mathcal{F}_q}^2 = \sum_n e^{-2nt} = (1-e^{-2t})^{-1}$.  Note from \eqref{eq: c.omega=e} that $c^n\Omega =e^{\otimes n}$; thus, the isometric isomorphism $X\mapsto X\Omega$ between $\mathcal{F}_q(\mathbb{C}\oplus\mathbb{C})$ and $L^2(\mathrm{W}^\ast(c),\tau)$ identifies the vector $\psi_t$ with
\[ h_t=\sum_{n=0}^{\infty}e^{-nt} \frac{c^n}{\sqrt{[n]_q!}}. \]
By definition of $e^{-tN_q}$, we have
\[ e^{-tN_q}h_t =\sum_{n=0}^{\infty}e^{-2nt} \frac{c^n}{\sqrt{[n]_q!}}=h_{2t}.\]
Hence, by the upper bound proved above, we find that $\|h_{2t}\| \le \beta t^{-1} \|h_{t}\|_2 = \beta t^{-1} \|\psi_{t}\|_{\mathcal{F}_q} < \infty$ for all $t>0$.  In particular, $h_t$ is actually in $\mathrm{W}^\ast(c)$, and so we may test the sharpness of the ultracontractivity upper bound on it.

To do so, we use the following simple estimate for the norm of $h=h_{2t}$.  Since $h^\ast h\ge 0$ is positive semidefinite in the $\mathrm{C}^\ast$-algebra $\mathrm{W}^\ast(c)$, the operator $\|h^\ast h\|1-h^\ast h$ is also $\ge 0$.  Thus, the trace (or any state) satisfies
\begin{equation} \label{eq:L4L2.estimate} \tau(h^\ast h h^\ast h) \le \tau(\|h^\ast h\| h^\ast h) = \|h\|^2 \tau(h^\ast h). \end{equation}
We have already calculated above that $\tau(h^\ast h) = \|h\|_2^2 = \|\psi_{2t}\|_{\mathcal{F}_q}^2 = (1-e^{-4t})^{-1}$, so to bound $\|e^{-tN_q} h_t\| = \|h_{2t}\| = \|h\|$ from below, it behooves us to compute
\[ \tau(h^\ast hh^\ast h) = \|h_{2t}^\ast h_{2t}\Omega\|_{\mathcal{F}_q}^2. \]

By Lemma \ref{qCircular}, we have for each $n \in \mathbb{Z}_{\ge 0}$ \[c^{*n}=\sum_{k=0}^n \binom{n}{k}_q\overline{a}^{n-k}a^{*k}\ \]
 where $\binom{n}{k}_q=\frac{[n]_q!}{[k]_q![n-k]_q!}$ is the $q$-binomial coefficient. By applying this to $h_{2t}^\ast h_{2t}\Omega$, we get
 \begin{align*}
h_{2t}^\ast h_{2t}\Omega&=\sum_{m,n=0}^{\infty}e^{-2(m+n)t}\frac{c^{*m}(e^{\otimes n})}{\sqrt{[m]_q![n]_q!}}\\
&=\sum_{m,n=0}^{\infty}\sum_{k=0}^{m}e^{-2(m+n)t}\binom{m}{k}_q\frac{\overline{a}^{m-k}a^{*k}}{\sqrt{[m]_q![n]_q!}} e^{\otimes n}\\
&=\sum_{m,n=0}^{\infty}\sum_{k=0}^{\min(m,n)}e^{-2(m+n)t}\binom{m}{k}_q\frac{[n]_q[n-1]_q\cdots [n-k+1]_q}{\sqrt{[m]_q![n]_q!}} \ \overline{e}^{\otimes m-k} \otimes e^{\otimes n-k}\\
&=\sum_{m,n=0}^{\infty}\sum_{k=0}^{\min(m, n)}e^{-2(m+n)t}\binom{m}{k}_q\binom{n}{k}_q[k]_q! \ \frac{\overline{e}^{\otimes m-k} \otimes e^{\otimes n-k}}{\sqrt{[m]_q![n]_q!}}. 
\end{align*}
Since $\mathbb{Z}_{\ge 0}^{3}\ni (m,n,k)\mapsto (m+k,n+k,k)\in \{(a,b,c)\in \mathbb{Z}_{\ge 0}^3|0\le c \le \min(a,b)\}$ is a bijection, we obtain
\begin{align*}
h_{2t}^\ast h_{2t}\Omega&=\sum_{m,n,k=0}^{\infty}e^{-2(m+n+2k)t}\binom{m+k}{k}_q\binom{n+k}{k}_q[k]_q! \ \frac{\overline{e}^{\otimes m} \otimes  e^{\otimes n}}{\sqrt{[m+k]_q![n+k]_q!}}.\\
\end{align*}
Since $\langle \overline{e}^{\otimes m }\otimes e^{\otimes n},\overline{e}^{\otimes m' }\otimes e^{\otimes n'}\rangle_q=\delta_{m,m'}\delta_{n,n'} [m]_q![n]_q!$, we can compute $\|h_{2t}^\ast h_{2t}\Omega\|_{\mathcal{F}_q}^2$ as
 \begin{align*} &\sum_{m,n=0}^{\infty}\left(\sum_{k=0}^{\infty}e^{-2(m+n+2k)t}\binom{m+k}{k}_q\binom{n+k}{k}_q\ \frac{[k]_q!\sqrt{[m]_q![n]_q!}}{\sqrt{[m+k]_q![n+k]_q!}}\right)^2 \\
&=\sum_{m,n=0}^{\infty}e^{-4(m+n)t}\left(\sum_{k=0}^{\infty}e^{-4kt}\sqrt{\binom{m+k}{k}_q}\sqrt{\binom{n+k}{k}_q}\right)^2. 
\end{align*}
By definition $[n]_q=\frac{1-q^n}{1-q}$ and
\begin{equation} \label{eq.binom.bound} \binom{n}{k}_q^{-1}=\frac{\prod_{i=1}^k(1-q^i)}{\prod_{i=1}^k(1-q^{n-i+1})} \le b_q <\infty \quad\text{where}\quad b_q=\prod_{i=1}^{\infty} \frac{1+|q|^i}{1-|q|^i}.
\end{equation}
Hence, we have the following estimate:
\begin{align*}
\|h_{2t}^\ast h_{2t}\Omega\|_{\mathcal{F}_q}^2 &=     \sum_{m,n=0}^{\infty}e^{-4(m+n)t}\left(\sum_{k=0}^{\infty}e^{-4kt}\sqrt{\binom{m+k}{k}_q}\sqrt{\binom{n+k}{k}_q}\right)^2 \\
&\ge b_q^{-2} \left(\sum_{m=0}^{\infty}e^{-4mt}\right)^4 =b_q^{-2}(1-e^{-4t})^{-4}.
\end{align*}
Hence, using \eqref{eq:L4L2.estimate}, we have
\[ \frac{\|e^{-tN_q}h_t\|^2}{\|h_t\|_2^2} \ge \frac{\|h_{2t}^\ast h_{2t}\Omega\|_{\mathcal{F}_q}^2}{\|h_{2t}\|_2^2\|h_t\|_2^2} \ge \frac{b_q^{-2}(1-e^{-4t})^{-4}}{(1-e^{-4t})^{-1}(1-e^{-2t})^{-1}}
= b_q^{-2}\frac{(1+e^{-2t})^{-3}}{(1-e^{-2t})^2}. \]
This is $\ge \alpha^2 t^{-2}$ with $\alpha = 1/4b_q$, proving the stated lower bound.
\end{proof}

We now conclude with the observation that a simplified form of the argument for the lower bound in Theorem \ref{thm.ultra} shows that our $q$-circular strong  Haagerup inequality in Theorem \ref{Mainresult} is sharp, in a particularly strong form.

\begin{cor} \label{cor.ultra} For $-1<q<1$, there is a constant $b_q = b_{|q|}<\infty$ (cf.\ \eqref{eq.binom.bound}) such that
\[ \|c^n\| \ge b_q^{-1}\sqrt{n+1}\|c^n\|_2. \]
\end{cor}

\begin{proof} Adopt the same conventions from the proof of Theorem \ref{thm.ultra}. Expand as above, $c^n=\sum_{k=0}^n \binom{n}{k}_q a^{n-k} \overline{a}^{*k}$, and apply it to the normalized tensor $\psi=\overline{e}^{\otimes n}/\sqrt{[n]_q!}$.
    \begin{align*}
       c^n \psi= \sum_{k=0}^n \binom{n}{k}_q  \frac{[n]_q\cdots [n-k+1]_q}{\sqrt{[n]_q!}} \ e^{\otimes n-k}\otimes \overline{e}^{\otimes n-k}.
    \end{align*}
    Since $\langle e^{\otimes n-k}\otimes \overline{e}^{\otimes n-k},e^{\otimes n-k'}\otimes \overline{e}^{\otimes n-k'}\rangle_q=\delta_{k,k'}([n-k]_q!)^2$, we have
    \begin{align*}
      \|c^n \psi\|^2_{\mathcal{F}_q(H)}&=\sum_{k=0}^{n}\binom{n}{k}_q^2 \frac{[n]_q^2\cdots [n-k+1]_q^2}{[n]_q!} ([n-k]_q!)^2\\
      &=\sum_{k=0}^{n}\binom{n}{k}_q^2  [n]_q! \\
      & \ge b_q^{-2}[n]_q! (n+1)
    \end{align*}
where $b_q$ is the constant defined in \eqref{eq.binom.bound}.  Thus we obtain
    \[ \|c^n\|\ge b^{-1}_q \sqrt{(n+1)[n]_q!}=b^{-1}_q \sqrt{n+1}\|c^n\|_2\]
\end{proof}

\subsection*{Acknowledgement}
We thank Zhiyuan Yang for his suggestions which improved the readability of this paper. We also thank Roland Speicher for helpful discussions.

\bibliographystyle{amsalpha}
\bibliography{reference.bib}

\end{document}